\documentclass[reqno]{amsart}

\usepackage{amssymb}
\usepackage{setspace}
\usepackage{mathrsfs}
\usepackage{graphicx}
\usepackage{enumerate}
\usepackage[colorlinks=false, urlcolor=blue, final]{hyperref}
\usepackage[all]{xy}
\usepackage{pdfsync}
\usepackage[normalem]{ulem}

\newtheorem{thm}{Theorem}[section]
\newtheorem{lem}[thm]{Lemma}
\newtheorem{cor}[thm]{Corollary}
\newtheorem{prop}[thm]{Proposition}

\theoremstyle{definition}
\newtheorem{defn}[thm]{Definition}

\theoremstyle{remark}
\newtheorem{rmk}[thm]{Remark}
\newtheorem{ex}[thm]{Example}

\allowdisplaybreaks

\DeclareMathOperator{\Imp}{IMP}
\DeclareMathOperator{\lt}{lt}
\DeclareMathOperator{\ind}{Ind}
\DeclareMathOperator{\id}{id}
\DeclareMathOperator{\supp}{supp}
\DeclareMathOperator{\spn}{span}
\DeclareMathOperator{\Br}{Br}
\DeclareMathOperator{\Fix}{Fix}
\DeclareMathOperator{\Rep}{Rep}

\newcommand{\Ind}[2]{\ind_{#1}^{#2}}
\newcommand{\R}{\mathbb{R}}
\newcommand{\C}{\mathbb{C}}
\newcommand{\Z}{\mathbb{Z}}
\newcommand{\Zb}{\mathscr{Z}}
\newcommand{\N}{\mathbb{N}}
\newcommand{\unit}{^{(0)}}
\newcommand{\two}{^{(2)}}
\newcommand{\A}{\mathscr{A}}
\newcommand{\B}{\mathscr{B}}
\newcommand{\inv}{^{-1}}
\newcommand{\linner}[3]{{{}_{{}_{#1}}}\!\left\langle #2, #3\right\rangle}
\newcommand{\rinner}[3]{\left\langle #2, #3\right\rangle_{{}_{#1}}}
\newcommand{\h}{\mathcal{H}}
\newcommand{\cspn}{\overline{\spn}}
\newcommand{\sidehat}{^{\wedge}}

\usepackage{color}

\begin{document}
 
\title[The Brauer semigroup]{The  Brauer semigroup of a
   groupoid and a symmetric imprimitivity theorem}

\author[J. H. Brown]{Jonathan Henry Brown}
\address{Department of Mathematics \& Statistics, University of
  Otago, P.O. Box 56, Dunedin 9054 New Zealand}
\email{jbrown@maths.otago.ac.nz}

\author[G. Goehle]{Geoff Goehle}
\address{Mathematics \& Computer Science Department, Stillwell 426,
  Western Carolina University, Cullowhee, NC 28723}
\email{grgoehle@email.wcu.edu}



\begin{abstract}
In this paper we define a monoid called the  Brauer
semigroup for a locally compact Hausdorff groupoid $E$  whose elements
consist of Morita equivalence classes of $E$-dynamical systems.  This
construction generalizes both the equivariant Brauer semigroup for
transformation groups and the Brauer group for a groupoid.
We show that groupoid equivalence induces an isomorphism of
 Brauer semigroups and that this isomorphism  preserves
the Morita equivalence classes of the respective crossed products,
thus generalizing Raeburn's symmetric imprimitivity theorem.  
\end{abstract}

\maketitle
%%%%%%%%%%%%%%%%%%%%%%%%%%%%%%%%%%%%%%%%%%%%%%%%%%%%%%%%%%%%%%%%%%%%%%%%%%%%%%%%%%%%%%%%%%%%%%%%%%%%%%%%%%%%%%%%%%%%%%%%%%%%%%%%%%%%%%%%%%%%%%%%%%%%%%%
\section{Introduction}

Let $\mathcal{G}$ and $\mathcal{H}$ be groups with commuting free and
proper actions on the left and right, respectively, of  a space $X$.
In \cite{Rie82}, Rieffel attributed to Green 
the useful observation that $C_0(X/\mathcal{H})\rtimes \mathcal{G}$
is Morita equivalent to $C_0(\mathcal{G}\backslash X)\rtimes
\mathcal{H}$.  One of the many applications of this result is that if
$X=\mathcal{G}$ and $\mathcal{H}$ is a closed subgroup of $\mathcal{G}$ then we
can induce representations from $\mathcal{H}$ to obtain
representations of $\mathcal{G}$.  In \cite{Rae88}, Raeburn proved the
symmetric imprimitivity theorem, a noncommutative version of the
result in \cite{Rie82} which gives a Morita equivalence between
crossed products of $\mathcal{G}$ and $\mathcal{H}$ on certain
$C^*$-algebras (where $\mathcal{G}$ and $\mathcal{H}$ are groups
acting freely and properly on $X$ as above).   Again Raeburn's result can be used to construct
representations induced from subgroups.   

In an effort to study the cohomology of the transformation group
$(\mathcal{G},T)$,  the authors in \cite{CKRW97} define a group called
the equivariant Brauer group $\Br_{\mathcal{G}}(T)$.  The elements of
this group are
$\mathcal{G}$-dynamical systems $(A,\alpha)$ where $A$ is a continuous
trace $C^*$-algebra with spectrum $T$ and the action induced by $\alpha$ on $T$
coincides with the given action.  In \cite{KRW96}, the authors  use
the equivariant Brauer group to provide an algebraic setting for Raeburn's
symmetric imprimitivity theorem.  That is, they show that if
$\mathcal{G}$ and $\mathcal{H}$ have commuting free and proper actions on a
space $X$, then there exists an isomorphism $\theta:
\Br_{\mathcal{G}}(X/\mathcal{H})\to
\Br_{\mathcal{H}}(\mathcal{G}\backslash X)$ such that if
$\theta([A,\alpha])=[B,\beta]$ then $A\rtimes_\alpha \mathcal{G}$ is
Morita equivalent to $B\rtimes_\beta \mathcal{H}$.   

These results inspired two distinct generalizations.  In 2000, the authors in 
\cite{aHRW00} extend the results in  \cite{KRW96} to find a monoid of all separable
$\mathcal{G}$ systems $(A,\alpha)$ with $A$ a $C_0(T)$-algebra where
the action induced by $\alpha$ on $T$ coincides with a given action.  This allows the authors to recover the full power of Raeburn's
symmetric imprimitivity theorem.  In \cite{KMRW98}, the authors  replace the transformation group
$(\mathcal{G}, T)$ in \cite{CKRW97} with a groupoid $E$ and define a Brauer group $\Br(E)$ for $E$ and point out how $\Br(E)$ can be used to study groupoid cohomology.  

Our goal  is to
combine these two generalizations into one overarching framework. 
To that end, for a second countable locally compact Hausdorff groupoid $E$ we define a monoid called the
\emph{Brauer semigroup} $S(E)$ consisting of equivariant
Morita equivalence classes of $E$-dynamical systems
(Definition~\ref{def brauer semi}).  The Brauer group $\Br(E)$
 embeds in $S(E)$ as the set of invertible elements. We show that if
$G$ and $H$ are groupoids, $X$ is a $(G,H)$-equivalence, and $G\ltimes
X\rtimes H$ is the associated transformation groupoid, then there
exists  an isomorphism 
\[
\nu^{X,H}:S(H)\to S(G\ltimes X\rtimes H)
\]
 such
that if $\nu^{X,H}([B,\beta])=[A,\omega]$ then $B\rtimes_\beta H$ is
Morita equivalent to $A\rtimes_{\omega} G\ltimes X\rtimes H$
(Theorem~\ref{thm iso semigroup}).  By symmetry we then get an
isomorphism $\nu^X: S(H)\to S(G)$ with the same property.  

At first glance, it may appear 
the hypothesis in Theorem~\ref{thm iso semigroup} 
that $X$ be a groupoid equivalence  is stronger than the hypotheses used in
Raeburn's symmetric imprimitivity theorem.  However, if $\mathcal{G}$ and
$\mathcal{H}$ are groups with commuting free and proper actions on the
left and right, respectively, of a space $X$, then $X$
provides a groupoid equivalence between the transformation group
groupoids $\mathcal{G}\ltimes X/\mathcal{H}$ and
$\mathcal{G}\backslash X \rtimes \mathcal{H}$.  Furthermore, for a
$C_0(X/\mathcal{H})$-algebra $A$, $A\rtimes \mathcal{G}$ is isomorphic
to $A\rtimes (\mathcal{G}\ltimes X/\mathcal{H})$ (similarly for
$C_0(\mathcal{G}\backslash X)$-algebras).\footnote{See \cite[Example
  1]{goe:mackey1} for the proof in the case where $A$ has Hausdorff spectrum.}
Thus for $[A,\alpha]\in
\Br_{\mathcal{G}}(X/\mathcal{H})$ and $[B,\beta]\in
\Br_{\mathcal{H}}(\mathcal{G}\backslash X)$, the crossed product 
$A\rtimes \mathcal{G}$ is Morita equivalent to $B\rtimes\mathcal{H}$
if and only if $A\rtimes (\mathcal{G}\ltimes X/\mathcal{H})$ is Morita
equivalent to $B\rtimes
(\mathcal{G}\backslash X \rtimes \mathcal{H})$.  Hence, Theorem~\ref{thm
  iso semigroup} recovers Raeburn's symmetric imprimitivity theorem in
the group case. 

We begin the paper with a review of some preliminary materials,
including upper semicontinuous bundles, groupoids, and imprimitivity
bundles (Section~\ref{sec:preliminaries}).
The generalized fixed point algebra for a groupoid dynamical system,
as defined in \cite{mep09},
will play a key role in constructing an inverse for $\nu^{X,H}$.
However, the fixed point algebra is defined abstractly in
\cite{mep09} and in order to perform our analysis 
we need to find a more concrete description. We
do this in Section~\ref{sec ind alg}.  More specifically, for a principal
proper groupoid $E$ and $(A,\alpha)$ an $E$-dynamical system, we
define an algebra of continuous sections $\Ind{E}{E\unit} (A,\alpha)$,
and show in Proposition~\ref{prop: gen fix ind} that $\Ind{E}{E\unit}
(A,\alpha)$ is equal to (not just isomorphic to) the generalized fixed
point algebra. 

Next, we introduce the  Brauer semigroup for a groupoid in
Section~\ref{sec: brauer}.  Much of the work in 
defining the  Brauer semigroup
was done in \cite{aHRW00} and \cite{KMRW98}, so Section~\ref{sec: brauer} merely outlines the construction.  Section~\ref{sec:main-theorem} contains the statement and proof of the main
result of the paper, Theorem~\ref{thm iso semigroup}.  To prove
Theorem~\ref{thm iso semigroup} we follow the outline in
\cite{aHRW00}.  However, the proofs in our setting are substantially
different and require significant analysis.  In
Section~\ref{section nu} we show $\nu^{X,H}$ is a homomorphism.  In
Section~\ref{Fix} we use the generalized fixed point algebras
described above to construct a map from
$S(G\ltimes X\rtimes H)$ to $S(H)$ and in Section~\ref{sec: iso equiv}
we show this map is an  inverse for $\nu^{X,H}.$  We then use the results of
\cite{mep09} to get Morita equivalence as follows.  Let
$(A,\omega)$ be a $G\ltimes X\rtimes H$-dynamical system.  Then the
transformation groupoid $G\ltimes X$ includes in $G\ltimes X\rtimes H$
and we can restrict $\omega$ to an action $\omega^G$ of $G\ltimes X$
on $A$. Since $G$ acts freely and properly on $X$, $G\ltimes X$ is a
principal and proper groupoid. Let $\Fix_G(A)$ be the generalized
fixed point algebra for the $\omega^G$ action.  We will show that
there is an action $\Fix_G(\omega)$ of $H$ on $\Fix_G(A)$ and that
$(\nu^{X,H})\inv([A,\omega]) = ([\Fix_G(A),\Fix_G(\omega)])$.  Next,  
\cite{mep09} gives an imprimitivity bimodule $Z$
between $A\rtimes_{r} (G\ltimes X)$ and $\Fix_G(A)$.  Since
$G\ltimes X$ is principal and proper, it is amenable \cite{A-DR00}, so
 $Z$ is an  $A\rtimes (G\ltimes X)-\Fix_G(A)$-imprimitivity bimodule.  We then analyze $Z$ to show it is equivariant
for the $H$ actions so that
$\Fix_G(A)\rtimes H$ and $(A\rtimes (G\ltimes X))\rtimes H$ are
Morita equivalent by \cite[Section~9.1]{MW08}. Finally, \cite{BGW12}
shows that $(A\rtimes (G\ltimes X))\rtimes H\cong A\rtimes (G\ltimes
X\rtimes H)$, giving the result.  
We prefer this approach to the one outlined in \cite[Section~9.2]{MW08} (which
generalizes the constructions in \cite{KMRW98}) because our approach
allows for induction in stages.  In any case, we show in Section
\ref{sec:constr-from-kmr} that when restricted to the Brauer group the
isomorphism $\nu^X$ is equal to the
isomorphism of  Brauer groups constructed in
\cite{KMRW98}.  

In a short appendix we answer a question raised in
\cite{mep09} by giving a fairly general condition that guarantees 
the generalized fixed point algebra of a proper groupoid dynamical
system is Morita equivalent to an ideal of the reduced crossed
product. 

%%%%%%%%%%%%%%%%%%%%%%%%%%%%%%%%%%%%%%%%%%%%%%%%%%%%%%%%%%%%%%%%%%%%%%%%%%%%%%%%%%%%%%%%%%%%%%%%%%%%%%%%%%%%%%%%%%%%%%%%%%%%%%%%%%%%%%%%%%%%%%%%%%%%%%%
\subsection*{Acknowledgments} We would like to thank Dana Williams for
his advice and support.  Part of the research for this
project was undertaken while the first author  was supported by a
postdoctoral fellowship funded by the 
Skirball Foundation via the Center for Advanced Studies in Mathematics
at Ben-Gurion University of the Negev. 

%%%%%%%%%%%%%%%%%%%%%%%%%%%%%%%%%%%%%%%%%%%%%%%%%%%%%%%%%%%%%%%%%%%%%%%%%%%%%%%%%%%%%%%%%%%%%%%%%%%%%%%%%%%%%%%%%%%%%%%%%%%%%%%%%%%%%%%%%%%%%%%%%%%%%%%

\section{Preliminaries}
\label{sec:preliminaries}

We assume  all Banach algebras and $C^*$-algebras (with the exception of multiplier algebras) are  separable.  For a $C^*$-algebra $A$   we denote the multiplier algebra of $A$ by $M(A)$ and its center by $Z(A)$.  

%%%%%%%%%%%%%%%%%%%%%%%%%%%%%%%%%%%%%%%%%%%%%%%%%%%%%%%%%%%%%%%%%%%%%%%%%%%%%%%%%%%%%%%%%%%%%%%%%%%%%%%%%%%%%%%%%%%%%%%%%%%%%%%%%%%%%%%%%%%%%%%%%%%%%%%

\subsection{Upper semicontinuous bundles}

Let $p:X\to T$ and $q: Y\to T$ be surjections.  Throughout we denote the fibered product of $X$ and $Y$ by
\[
X*Y:=\{(x,y)\in X\times Y: p(x)=q(y)\}.
\] 

\begin{defn}[{\cite[Definition~3.1]{MW08}}]
\label{def usc}
Let $T$ be a second countable locally compact Hausdorff space.  An \emph{upper semicontinuous Banach bundle over $T$} is a topological space $\Zb$ together with a continuous open surjection $p_\Zb: \Zb\to T$ such that $Z(t):=p_\Zb\inv(t)$ is a Banach space for each $t\in T$ and such that the following axioms hold.
\begin{enumerate}
\item \label{it: usc} The map $z\mapsto \|z\|$ is upper semicontinuous
  from $\Zb$ to $\R^+$. 
\item \label{it: add} The map $\Zb*\Zb\to \Zb$ defined by
  $(z,w)\mapsto z+w$ is continuous. 
\item \label{it: scal} For each $\kappa\in \C$ the map $\Zb\to \Zb$
  defined by $z\mapsto \kappa z$ is continuous. 
\item \label{it: conv} If $\{z_i\}$ is a net in $\Zb$ such that
  $\|z_i\|\to 0$ and $p_\Zb(z_i)\to t$ then $z_i\to 0_t$ where $0_t$
  is the zero element in $Z(t)$. 
\end{enumerate}
An \emph{upper semicontinuous $C^*$-bundle} over $T$ is
an upper semicontinuous Banach bundle $p_\Zb:\Zb \to T$ such that
$Z(t)$ is a $C^*$-algebra for each $t$ and the following additional
axioms hold.
\begin{enumerate}
\item[(5)] The map $\Zb*\Zb\to \Zb$ defined by $(z,w)\mapsto zw$ is
  continuous. 
\item[(6)] The map $\Zb\to \Zb$ defined by $z\mapsto z^*$ is continuous.
\end{enumerate}
\end{defn}

Showing a sequence in an upper semicontinuous Banach bundle converges
is often very delicate.  Our main tool for this is the following proposition 
 which roughly states that a sequence is  
convergent if there exists a convergent sequence close to
it. Many of 
our proofs amount to finding a convergent sequence close to a given
sequence.    The proof is the same as in \cite[Proposition~C.20]{TFB2} so we omit it. 

\begin{prop}
\label{Proposition c.20}
Let $p_\Zb:\Zb\to T$ be an upper semicontinuous Banach bundle over
$T$.  Let $\{a_i\}_{i\in I}$ be a net in $\Zb$ such that
$p_\Zb(a_i)\to p_\Zb(a)$ for some $a\in \Zb$.  Suppose that for all
$\epsilon>0$ there is a net $\{b_i\}_{i\in I}$ and $b\in \Z$ such that 
\begin{enumerate}
\item $b_i\to b$ in $\Zb$,
\item $p_\Zb(b_i)=p_\Zb(a_i)$ and $p_\Zb(b)=p_\Zb(a)$, and 
\item $\max\{\|a-b\|, \|a_i-b_i\|\}<\epsilon$ eventually.
\end{enumerate}
Then $a_i\to a$.
\end{prop}

Let $p_\Zb:\Zb\to T$ be an upper semicontinuous Banach bundle over $T$
and $q:X\to T$ be a continuous open surjection.  We  define the pull
back bundle over $X$ by  
\[
q^*\Zb:=\{(x,z)\in X\times\Zb: q(x)=p_\Zb(z)\}\quad\text{with}\ p_{q^*\Zb}: (x,z)\to x
\]
where the topology on $q^*\Zb$ is given by the relative topology. Note that the fiber of $q^*\Zb$ over $x$ is naturally isomorphic to $Z(q(x))$.

Let $p_\Zb: \Zb\to T$ be an upper semicontinuous Banach bundle.  A
continuous function $f: T\to \Zb$ is called a \emph{section} if
$p_\Zb\circ f=\id_T$.  We denote the set of  continuous sections by
$\Gamma(T,\Zb)$, the continuous bounded sections by $\Gamma^b(T,\Zb)$,
the continuous compactly supported sections by $\Gamma_c(T,\Zb)$, and
the continuous sections that vanish at infinity by $\Gamma_0(T,
\Zb)$. By the Tietze Extension Theorem for upper semicontinuous
Banach bundles \cite[Proposition~A.5]{MW08Fell}, if $T$ is locally
compact Hausdorff then $\{f(t): f\in \Gamma_0(T,\Zb)\}=Z(t)$.  Since
we only consider bundles over locally compact Hausdorff spaces  we
will use this property without comment.

The equation $\|f\|=\sup_{t\in T}\|f(t)\|$ defines a norm on
$Z=\Gamma_0(T, \Zb)$   and under this norm $Z$ is a Banach algebra.
It is a $C^*$-algebra if $\Zb$ is an upper semicontinuous
$C^*$-bundle.  In either case we refer to $\Gamma_0(T,\Zb)$ as the
section algebra of $\Zb$ and denote it by the corresponding Roman
letter $Z$.  For $\phi\in
C_0(T)$ and $f\in \Gamma(T,\Zb)$ define
\[
\phi\cdot f(t)=\phi(t) f(t).
\]
Observe that if $\phi\in C_c(T)$ then $\phi\cdot f\in \Gamma_c(T,\Zb)$.  Further,
the map $\phi\mapsto (f\mapsto \phi\cdot f)$ induces an action of $C_0(T)$ on
$Z$ \cite[Lemma~C.22]{TFB2}.

Let $p_\Zb:\Zb\to T$ and $p_{\mathscr{Y}}:\mathscr{Y}\to T$ be upper
semicontinuous Banach bundles over $T$. We say $\Phi: \Zb\to
\mathscr{Y}$ is a homomorphism if $\Phi$ is continuous,
$p_\Zb(z)=p_{\mathscr{Y}}(\Phi(z))$, and $\Phi$ is a homomorphism on
the fibres.  A bundle homomorphism  
$\Phi: \Zb\to \mathscr{Y}$ induces a $C_0(T)$-linear homomorphism
$f\mapsto (t\mapsto \Phi(f)(t))$ of the section algebras.  
Every $C_0(T)$-linear homomorphism $Z\to Y$ induces a
bundle homomorphism of $\Zb\to \mathscr{Y}$ as well \cite[page
18]{MW08}.
We will often convert from bundle homomorphisms to $C_0(T)$-linear
homomorphisms without comment.

\begin{defn}
\label{def: c0x alg}
Let $T$ be a second countable locally compact Hausdorff space and $A$ be a (separable) $C^*$-algebra. We say
 $A$ is a \emph{$C_0(T)$-algebra} if there exists a
nondegenerate $*$-homomorphism of $C_0(T)$ into the center of the
multiplier algebra of $A$. 
\end{defn}

If $\A$ is an upper semicontinuous $C^*$-bundle over $T$ then the
section algebra $A=\Gamma_0(T, \A)$ is a $C_0(T)$-algebra.  In fact
all $C_0(T)$-algebras arise in this way. 

\begin{prop}[{\cite[Theorem~C.26]{TFB2}}]
\label{prop: c0t alg}
Let $T$ be a second countable locally compact Hausdorff space and $A$ a $C^*$-algebra with spectrum $\widehat{A}$.  Then the following are equivalent:
\begin{enumerate}
\item $A$ is a $C_0(T)$-algebra.
\item There exists an upper semicontinuous $C^*$-bundle $\A$ over $T$ such that $A=\Gamma_0(T,\A)$.
\item There exists a continuous surjection $\sigma_A: \widehat{A}\to T$.
\end{enumerate}
\end{prop}

The next two propositions appear in \cite[Corollary~II.14.7 and
Theorem~II.13.18]{FD88n1} for Banach bundles and are proven in
\cite[Proposition~C.24 and Theorem~C.25]{TFB2} for upper
semicontinuous $C^*$-bundles.   We restate them here for the
convenience of the reader. The proofs in \cite{TFB2} go through 
without change for upper semicontinuous Banach bundles.   

\begin{prop}
\label{prop: when dense}
Let $p_\Zb: \Zb\to T$ be an upper semicontinuous Banach bundle and
$\Gamma$ a subspace of $\Gamma_0(T,\Zb)$.  Suppose  
\begin{enumerate}
\item\label{dense alg} $f\in \Gamma$ and $\phi\in C_0(T)$ implies $\phi\cdot f\in \Gamma$, and
\item\label{dense fib} for each $t\in T$, $\{f(t): f\in \Gamma\}$ is dense in $Z(t)$.
\end{enumerate}
Then $\Gamma$ is dense in $\Gamma_0(T, \Zb)$.
\end{prop}

\begin{prop}
\label{prop: defn a bundle}
Let $\Zb$ be a set and $p_{\Zb}: \Zb\to T$ a surjection onto a second countable
locally compact Hausdorff space $T$ such that $Z(t)$ is a Banach
space.  Suppose $\Gamma$ is an algebra of sections of $\Zb$ such that  
\begin{enumerate}
\item \label{it bund usc}for each $f\in \Gamma$, $t\mapsto \|f(t)\|$ is upper semicontinuous, and 
\item \label{it bund dens} for each $t\in T$, $\{f(t):f\in \Gamma\}$
  is dense in $Z(t)$. 
\end{enumerate}
Then there is a unique topology on $\Zb$ such that $p_{\Zb}:\Zb\to T$
is an upper semicontinuous Banach bundle over $T$ with $\Gamma\subset
\Gamma(T, \Zb)$. 
If we replace ``Banach space'' with $C^*$-algebra then $\Zb$ is an
upper semicontinuous $C^*$-bundle.  
\end{prop}

Proposition~\ref{prop: c0t alg} states  that the map $\A\mapsto
\Gamma_0(T,\A)$ defines a one-to-one correspondence between upper
semicontinuous $C^*$-bundles and $C_0(T)$-algebras. In the following
we establish this correspondence for imprimitivity bimodules over $T$.
The proof follows similar lines to the proof of
\cite[Theorem~C.26]{TFB2}. 

Suppose $A$ and $B$ are $C_0(T)$-algebras and $Z$ is an
$A-B$-imprimitivity bimodule.   Then the actions of $C_0(T)$ on $A$
and $B$ induce left and right actions of $C_0(T)$ on $Z$.  If $f\cdot
z=z\cdot f$ for all $f\in C_0(T)$ and all $z\in Z$ we say that $Z$ is
an $A-B$ imprimitivity bimodule over $T$.   

Let $Z$ be an $A-B$ imprimitivity bimodule over $T$.  Define
$C_{0,t}(T):=\{f\in C_0(T): f(t)=0\}$ and consider
$M_t:=C_{0,t}(T)\cdot Z:=\cspn\{f\cdot z: f\in C_{0,t}(T), z\in Z\}$.
For $z-w\in M_t$ we write $z\sim_t w$ and this turns
out to be an equivalence relation on $Z$. Define $Z(t) := Z/\sim_t$
and let $q_t$ be the quotient map.
For $z\in Z$ define $z(t):=q_t(z)$.  The quotient $Z(t)$ is an
$A(t)-B(t)$-imprimitivity bimodule whose actions and inner products
are characterized by 
\begin{align*}
\linner{A(t)}{z(t)}{w(t)}&=\linner{A}{z}{w}(t)&
\rinner{B(t)}{z(t)}{w(t)}&=\rinner{B}{z}{w}(t)\\ 
a(t)\cdot z(t)&=(a\cdot z)(t) & z(t)\cdot b(t) &=(z\cdot b)(t)
\end{align*}
where $a\in A$, $b\in B$ and $z,w\in Z$.  Define $\Zb:=\bigsqcup
Z(t)$ and $p_{\Zb}:\Zb\to T$ to be the obvious map. 

Consider the set of functions $\Gamma=\{t\mapsto z(t): z\in Z\}$.
Then for each $t\in T$, $\{z(t): z\in \Gamma\}=Z(t)$. 
Furthermore, since $\|z(t)\|:=\sqrt{\|\linner{A(t)}{z(t)}{z(t)}\|}=
\sqrt{\|\linner{A}{z}{z}(t)\|}$
and since $\linner{A}{z}{z}$ is in the $C_0(T)$-algebra $A$, we have
$t\mapsto \|z(t)\|$ is upper
semicontinuous.  Thus by Proposition~\ref{prop: defn a bundle} there
is a unique topology on $\Zb$ making it an upper semicontinuous Banach
bundle such that $z\mapsto z(t)$  is a continuous section for all
$z\in Z$.  This section vanishes at infinity since $t\mapsto
\linner{A}{z}{z}(t)$ does.  Note that since $\linner{A}{z}{w}\in A$,
$t\mapsto \linner{A}{z}{w}(t)$ is continuous. 

\begin{defn}
\label{def im bi bund}
Let $p_\A:\A\to T$ and $p_\B:\B\to T$ be $C^*$-bundles with section
algebras $A$ and $B$ respectively. A Banach bundle $p_\Zb:\Zb\to T$ is
an \emph{$\A-\B$-imprimitivity bimodule bundle} if each fibre $Z(t)$
is an $A(t)-B(t)$-imprimitivity bimodule such that the actions
$(a,z)\to a\cdot z$ from $\A*\Zb$ to $\Zb$, $(z,b)\to z\cdot b$ from
$\Zb*\B$ to $\Zb$, and inner products $(z,w)\mapsto
\linner{A(p_\Zb(z))}{z}{w}$ from $\Zb*\Zb$ to $\A$, and $(z,w)\mapsto
\rinner{B(p_\Zb(z))}{z}{w}$  from $\Zb*\Zb$ to $\B$ are continuous.   
\end{defn}

\begin{rmk}
\label{rmk: reconcile with KPRW}
Definition~\ref{def im bi bund} is slightly different to that used in
\cite[Definition~2.17]{KMRW98}.  In \cite{KMRW98} the authors do not
assume that the inner products are continuous.  The continuity of the
inner products in \cite{KMRW98} is implied by the continuity of the norm
on the bundle.  Since we only have upper semicontinuous $C^*$-bundles
we need to assume that the inner products are continuous.  This small
difference in our definition means that when showing a bundle is an
imprimitivity bimodule bundle we often only need to check the
continuity of inner products as the other conditions were checked in
\cite{KMRW98}. 
\end{rmk}

The next proposition is a slight generalization of \cite[Proposition~2.18]{KMRW98}.  We proved one direction above, the other follows exactly as it does in \cite{KMRW98} so we omit it.

\begin{prop}
\label{prop: bundle corr}
If $\Zb$ is an $\A-\B$-imprimitivity bimodule bundle then
$Z=\Gamma_0(T,\Zb)$ is an $A-B$ imprimitivity bimodule over $T$.
Conversely, if $Z$ is an $A-B$ imprimitivity bimodule over $T$
then there exists a unique $\A-\B$ imprimitivity bimodule bundle $\Zb$
such that $Z=\Gamma_0(T,\Zb)$. 
\end{prop}

%%%%%%%%%%%%%%%%%%%%%%%%%%%%%%%%%%%%%%%%%%%%%%%%%%%%%%%%%%%%%%%%%%%%%%%%%%%%%%%%%%%%%%%%%%%%%%%%%%%%%%%%%%%%%%%%%%%%%%%%%%%%%%%%%%%%%%%%%%%%%%%%%%%%%%%

\subsection{Groupoids}

A groupoid is a small category in which every morphism is invertible.
We say a groupoid $E$ is second countable locally compact Hausdorff if
it has a second countable locally compact Hausdorff topology in which
composition and inversion are continuous.  We assume all groupoids are
second countable locally compact and Hausdorff.  The objects of $E$
can be identified with the identity morphisms.  We refer to the set of
identity morphisms as the unit space, denoted $E\unit$, and elements
of $E\unit$ as units.
There are two natural continuous surjections $r_E,s_E: E\to E\unit$
given by $r_E(\gamma)=\gamma\gamma\inv$ and $s_E(\gamma)=\gamma\inv
\gamma$. We drop the subscript from the notation when the domain is
clear from context. For $u\in E\unit$ we denote $E^u:=r\inv(u)$ and
$E_u:=s\inv (u)$ and for $D$ a subset of $E\unit$ we denote
$E|_D:=\{\gamma\in E: r(\gamma),s(\gamma)\in D\}$.  It is
straightforward to check $E|_D$ is a subgroupoid of $E$. 

We say a groupoid $E$ acts on the left of a space
$X$ if there exists a continuous open surjection $r_X:X\to E\unit$
and a continuous map $E*X\to X$  given
by $(\gamma, x)\mapsto \gamma x$ such that $r_X(\gamma x)=r_E(\gamma)$
and $\gamma(\eta x) = (\gamma\eta)x$ for composable $\gamma$ and
$\eta$.\footnote{Since $\gamma$ and $\eta$ are composable only if
  $s(\gamma)=r(\eta)$, the relation $\gamma(\eta x) = (\gamma\eta)x$
  shows that $E$ can only act on spaces fibred over $E\unit$.} 
The definition of a right action $X*E \to X : 
(x,\gamma)\mapsto x\gamma$ is analogous. We will
use $E\cdot x$ to denote both the image of $x$ in $E\backslash X$
as well as the orbit of $x$ in $X$.  If $r_E$ is open and $E$
acts on $X$ then the quotient map $X\to E\backslash X$ is
open \cite[Lemma~2.1]{MW95}.

An action of $E$ on $X$ is
\emph{principal} (or free) if $\gamma x=x$ implies $\gamma=r_X(x)$.
An action is \emph{proper} if the set $\{\gamma\in E: \gamma K\cap
L\neq \emptyset\}$ is compact for all compact subsets $K$ and $L$ of
$X$. If the action of $E$ on $X$ is proper then the quotient 
space $E\backslash X$ is locally compact Hausdorff 
\cite[Proposition 2.1.12]{A-DR00}.

Note that  $r_E$ is open if and only if $s_E$ is open.  In this case
$E$ acts on the left and right of $E\unit$ by $\gamma\cdot
s(\gamma):=r(\gamma)$ and $r(\gamma)\cdot \gamma=s(\gamma)$.  We say
$E$ is principal if this action is principal; we say $E$ is proper if
this action is proper.  The orbit of a unit $u$ under this action is
then $E\cdot u:=\{r(\gamma): s(\gamma)=u\}$.  

Throughout we assume that our groupoids come equipped with a Haar
system.  That is, a system of measures $\{\lambda^u\}_{u\in E\unit}$ such that 
\begin{enumerate}
\item\label{haar full} $\supp(\lambda^u)=E^u$,
\item\label{haar cont} $u\mapsto \int_E f(\gamma)\;d\lambda^u(\gamma)$
  is continuous for all $f\in C_c(E)$, and
\item\label{haar left inv} $\int_E
  f(\eta\gamma)\;d\lambda^{s(\eta)}(\gamma)=\int_E
  f(\gamma)\;d\lambda^{r(\eta)}(\gamma)$. 
\end{enumerate}
If $E$ has a Haar system then $r$ and $s$ are open \cite[Corollary~page
118]{Sed86}. Note that condition \eqref{haar cont} implies that
$\sup_{u\in E\unit} \lambda^u(K)<\infty $ for all compact $K$ in $E$.

Given a left action of $E$ on $X$ we define the transformation
groupoid to be $E\ltimes X:=\{(\gamma,x): r_E(\gamma)=r_X(x)\}$ with unit
space $X$ and range and source maps $r(\gamma,x)=x$ and
$s(\gamma,x)=\gamma\inv x$.  If $E$ has a Haar system
$\{\lambda^u\}_{u\in E\unit}$ then
the set $\{\lambda^{r_X(x)}\times \delta_x\}_{x\in X}$ forms a Haar system for
$E\ltimes X$.  We can construct a transformation groupoid $X\rtimes E$
from a right action in a similar fashion.  

\begin{defn}
Let $E$ be a second countable locally compact Hausdorff groupoid with unit space
$E\unit$ and $\Zb$ an upper semicontinuous Banach bundle over $E\unit$
with (separable) section algebra $Z$.  We say $E$ acts on $Z$ if for each
$\gamma\in E$ there exists a norm preserving  isomorphism $V_\gamma:
Z(s(\gamma))\to Z(r(\gamma))$ such that  
\begin{enumerate}
\item \label{it: act hom} $V_\gamma V_\eta=V_{\gamma\eta}$ for all
  $(\gamma,\eta)\in E\two$ and 
\item \label{it: act cont} the map $E*\Zb\to \Zb$ defined by
  $(\gamma,z)\mapsto V_{\gamma}(z)$ is continuous. 
\end{enumerate}
If $\Zb$ is an upper semicontinuous $C^*$-bundle and $V_\gamma$ is a
$*$-isomorphism for all $\gamma$ then we refer to the pair $(Z, V)$ as
an \emph{$E$-dynamical system}. 
\end{defn}

Let $(A,\alpha)$ be an $E$-dynamical system. Consider $\Gamma_c(E,
r^*\A)$.  By \cite[Proposition~4.4]{MW08} the formulas 
\begin{align*}
f*g(\gamma):=\int_E f(\eta)\alpha_{\eta}(g(\eta\inv
\gamma))\;d\lambda^{r(\gamma)}(\eta)\quad\text{and}\quad
f^*(\gamma):=\alpha_\gamma(f(\gamma\inv)^*) 
\end{align*}
define a $*$-algebra structure on $\Gamma_c(E, r^*\A)$.  We define a
norm on $\Gamma_c(E, r^*\A)$ by  
\[
\|f\|_I:=\max\left\{\sup_{u\in E\unit}\int_E
  \|f(\gamma)\|d\lambda^u(\gamma), \sup_{u\in E\unit}
  \int_E \|f(\gamma\inv)\|d\lambda^u(\gamma)\right\}.
\]
Let $\Rep(E,A)$ be the set of $I$-norm bounded representations of
$\Gamma_c(E,r^*\A)$. We then define the crossed product
$A\rtimes_\alpha E$ to be the completion of $\Gamma_c(E, r^*\A)$ under
the norm $\|f\|=\sup\{\|\pi(f)\|: \pi\in \Rep(E,A)\}$.  The reduced
crossed product $A\rtimes_{\alpha,r} E$ is the completion of
$\Gamma_c(E,r^*\A)$ under the norm induced by ``regular
representations'' \cite[Section~2.2]{mep09}. The crossed products
considered in this paper  involve ``amenable'' groupoids and in this
case the reduced crossed product coincides with the crossed product.

\begin{ex}
\label{ex: trivial}
Let $E$ be a groupoid.  Then $C_0(E\unit)$ is a $C_0(E\unit)$-algebra.
The corresponding upper semicontinuous $C^*$-bundle is $E\unit\times
\C$.  For each $\gamma\in E$ define $\lt_\gamma(s(\gamma),
\kappa) = (r(\gamma),\kappa)$.  Then $\lt$ is a continuous action of
$E$ on $C_0(E\unit)$ called \emph{left translation}.  The resulting
crossed product $C_0(E\unit)\rtimes_{\lt} E$ is isomorphic to $C^*(E)$
\cite[Remark~4.22]{goe:unitary}. 
\end{ex}

%%%%%%%%%%%%%%%%%%%%%%%%%%%%%%%%%%%%%%%%%%%%%%%%%%%%%%%%%%%%%%%%%%%%%%%%%%%%%%%%%%%%%%%%%%%%%%%%%%%%%%%%%%%%%%%%%%%%%%%%%%%%%%%%%%%%%%%%%%%%%%%%%%%%%%%

\section{Induced algebras}
\label{sec ind alg}
\begin{defn}
\label{def: induced alg}
Let $E$ be a principal and proper groupoid with a Haar
system, $\Zb$ an upper semicontinuous Banach bundle over $E\unit$,
$Z=\Gamma_0(E\unit, \Zb)$, and $V=\{V_\gamma\}_{\gamma\in E}$ a
continuous action of $E$ on $\Zb$.  Define  
\begin{equation}
\Ind{E}{E\unit}(Z,V):=\left\{f\in \Gamma^b(E\unit,\Zb): 
\begin{array}{l} \text{$f(r(\gamma))=V_{\gamma}(f(s(\gamma)))$ and} \\
\text{$E\cdot u\mapsto \|f(u)\|$ vanishes at
$\infty$} \end{array}\right\}.
\end{equation}
We denote $\Ind{E}{E\unit}(Z,V)$ by $\Ind{}{}(Z,V)$ or just
$\Ind{}{}(Z)$ when clear from context. 
\end{defn}

Throughout this section let $E$ be a principal and proper groupoid with
Haar system $\{\lambda^u\}_{u\in E\unit}$, and  $p_\Zb: \Zb\to E\unit$ an upper semicontinuous
Banach bundle.  

\begin{lem}
\label{lem: def lambda}
For  $h\in \Gamma_c(E\unit, \Zb)$ the map
\begin{equation}
\label{eq:1}
\lambda(
h)(u):=\int_E V_\gamma(h(s(\gamma)))\;d\lambda^u(\gamma)
\end{equation}
is a well-defined element of $\Ind{E}{E\unit}(Z,V)$.
\end{lem}

\begin{proof}
Since $E$ is proper the set $E^v\cap s\inv (\supp(h))$ is compact
for each $v\in E\unit$.  Hence \eqref{eq:1} is defined for all $v$ and 
\begin{equation}
\label{supp lambda}
\supp (\lambda( h))\subset E\cdot \supp (h).
\end{equation}
Because $\lambda^u$ is a Haar system it is not hard to prove that
$\lambda(h)$ is continuous.  We show that $\lambda(
h)\in \Ind{E}{E\unit}(Z,V)$.  It follows from a brief computation that
$V_\eta(\lambda(h)(s(\eta))) = \lambda(h)(r(\eta))$.  
Thus $E\cdot v\mapsto \|\lambda(h)(v)\|$ is well-defined and
by \eqref{supp lambda} it has compact support contained in the image
of $\supp(h)$ under the quotient map.  Since the upper
  semicontinuous image of a compact set is bounded above, this implies
  that $\|\lambda( h)(v)\|$ is bounded and therefore
  $\lambda( h)\in\Ind{}{}(Z,V)$. 
\end{proof}

\begin{lem}
 \label{claim: fibre dense}
 Let $g\in \Gamma^b(E\unit, \Zb)$, $u\in E\unit$,  and $\epsilon>0$. 
  \begin{enumerate}
  \item\label{fd 1} Then for any $h\in \Gamma_c(E\unit, \Zb)$ such
    that $h(u)=g(u)$ there exists $\psi\in C_c(E\unit)$ such that
    $\lambda(\psi\cdot h)\in \Ind{}{}(Z,V)$ and
    $\|\lambda(\psi\cdot h)(u)-g(u)\|<\epsilon$. 
  \item\label{fd 2} For any $z\in \Zb$ there exists an $f\in
    \Ind{}{}(Z,V)$ such that $\|f(p_\Zb(z))-z\|<\epsilon$.  
  \end{enumerate}
  \end{lem}

\begin{proof} 
For item \eqref{fd 1}, let $g\in \Gamma^b(E\unit, \Zb)$, $u\in E\unit$,
$a=g(u)$,  $\epsilon>0$, and pick $h\in \Gamma_c(E\unit,\Zb)$ such
that $h(u)=a$.   Then the map  
\[
\gamma\mapsto V_\gamma(h(s(\gamma)))-h(r(\gamma))
\] 
is continuous.  Since the norm is upper semicontinuous, the set
\[
N_\epsilon:=\{\gamma\in
E:\|V_\gamma(h(s(\gamma)))-h(r(\gamma))\|<\epsilon\}
\]
is open.  Because $E$ is principal and proper we can apply
\cite[Lemma~5.3]{mep09} to find an open neighborhood $U\subset E\unit$
of $u$ such that  
\(
\{\gamma:\gamma\cdot U\cap U\neq \emptyset\}\subset N_\epsilon\cap
N_\epsilon\inv.
\)
 Pick a function $\phi\in C_c(E\unit)$  such that $0\leq \phi \leq 1$,
 $\phi(u)=1$, and $\supp(\phi)\subset U$.  Define 
\[
\psi(v)= \left(\int_E \phi(s(\eta))\;d\lambda^u(\eta)\right)\inv \phi(v)
\]
and note that
\begin{align*}
\int_E\psi(\gamma)\;d\lambda^{u}(\gamma)&=\int_E
\left(\int_E \phi(s(\eta)) \;d\lambda^{u}(\eta)\right)\inv
\phi(s(\gamma))\;d\lambda^{u}(\gamma) = 1.
\end{align*}
By Lemma~\ref{lem: def lambda}, $\lambda(\psi\cdot
h)\in\Ind{}{}(Z,V)$. It remains to show that $\|\lambda(\psi\cdot
h)(u)-a\|<\epsilon.$  We compute 
\begin{align*}
\|\lambda(\psi\cdot h)(u)-a\|&=\left\|\int_E
  \psi(s(\gamma))V_\gamma(h(s(\gamma)))\;d\lambda^{u}(\gamma)-h(u)\right\| 
\intertext{which, since $\int_E \psi(s(\gamma))  \;d\lambda^{u}(\gamma)=1$,} 
&=\left\|\int_E
  \psi(s(\gamma))V_\gamma(h(s(\gamma)))\;d\lambda^{u}(\gamma)-\int_E\psi(s(\gamma))
  \;d\lambda^{u}(\gamma)h(u)\right\|\\ 
&\leq \int_E
\psi(s(\gamma))\|V_\gamma(h(s(\gamma)))-h(u)\|\;d\lambda^u(\gamma).
\end{align*}
However, $u\in U$ and $s(\gamma\inv)\in \supp(\psi)\subset U$
  implies that $\gamma\in N_\epsilon$ so 
\[
\|\lambda(\psi\cdot h)(u)-a\|
<\epsilon \int_E \psi(s(\gamma))\;d\lambda^{u}(\gamma)=\epsilon.
\]
For item \eqref{fd 2} pick $h\in \Gamma_c(E\unit, \Zb)$ such that
$h(p_{\Zb}(z))=z$.  By the first part there exists
$\psi\in C_c(E\unit)$ such that $\lambda(\psi\cdot h)\in
\Ind{}{}(Z,V)$ and $\|\lambda(\psi\cdot
h)(p_{\Zb}(z))-z\|<\epsilon$. Then $f=\lambda(\psi\cdot h)$ suffices. 
\end{proof}

Let $(A, \alpha)$ be an $E$-dynamical system. For $u\in E\unit$ let
$\varepsilon_u: \Gamma^b(E\unit,\A)\to A(u)$ be defined by evaluation;
$\varepsilon_u(f) = f(u)$. 

\begin{prop}
\label{lem:  ind fix}
 Let $E$ be a  principal and proper groupoid with a Haar
 system and $(A, \alpha)$ an $E$-dynamical system. 

\begin{enumerate}
\item\label{Mphi}
For $\phi\in C_0(E\backslash E\unit)$ and $f\in
\Ind{E}{E\unit}(A,\alpha)$ define 
\[
Q_\phi(f): u\mapsto \phi(E\cdot u)f(u).
\]
Then $\phi\mapsto Q_\phi$ defines a $C_0(E\backslash E\unit)$-algebra
structure on $\Ind{E}{E\unit}(A,\alpha)$. 
\item\label{R} For $f\in \Ind{E}{E\unit}(A,\alpha)$ the map  
 \[
 \tilde{R}:f\mapsto f|_{E\cdot u}
 \]
induces an isomorphism 
$R:\Ind{E}{E\unit}(A,\alpha)(E\cdot u)\to\Ind{E|_{E\cdot u}}{E\cdot
  u}(\Gamma_0(E\cdot u, \A),\alpha)$.
\item\label{evaluation} The
evaluation map $\varepsilon_u$ factors to an isomorphism of
$\Ind{E}{E\unit}(A,\alpha)(E\cdot u)$ with $A(u)$.
\end{enumerate}
\end{prop}

\begin{proof}
For item \eqref{Mphi} first observe that $\Ind{}{}(A,\alpha)$ is a closed $*$-subalgebra of
the $C^*$-algebra $\Gamma^b(E\unit,\A)$ and hence is also a
$C^*$-algebra. Next  we prove that $\Ind{}{}(A,\alpha)$ is a
$C_0(E\backslash E\unit)$-algebra.  We show that $Q$ is a nondegenerate 
$*$-homomorphism from $C_0(E\backslash E\unit)$ to $Z(M(\Ind{}{}(A,\alpha)))$.  

We start by demonstrating that $Q_\phi(f)\in \Ind{}{}(A,\alpha)$ for
all $f\in \Ind{}{}(A,\alpha)$.  It follows from straightforward
computations that 
\begin{align*}
\alpha_\gamma(Q_\phi(f)(s(\gamma))) &= Q_\phi(f)(r(\gamma)) \quad
\text{and} \quad \|Q_\phi(f)(u)\| \leq \|\phi\|_\infty \|f(u)\|.
\end{align*}
This shows both that $Q_\phi$ is bounded by $\|\phi\|_\infty$ and that
the map $E\cdot u\mapsto \|Q_\phi(f)(u)\|$ vanishes at infinity.  Thus
$Q_\phi(f)\in \Ind{}{}(A,\alpha)$ and $Q_\phi$ is a bounded operator
on $\Ind{}{}(A,\alpha)$.  Simple calculations show that $Q_\phi$ is
linear and adjointable with adjoint $Q_{\overline{\phi}}$.
Thus $Q_\phi\in M(\Ind{}{}(A,\alpha))$. Since $\phi$ is scalar
valued, $Q_\phi(f)g(u)=\phi(E\cdot u)f(u)g(u)=fQ_\phi(g)(u)$ so
$Q_\phi\in  Z(M(\Ind{}{}(A,\alpha)))$ as desired.  More simple
computations show that $\phi\mapsto Q_\phi$ is a homomorphism.  

To show $\Ind{}{}(A,\alpha)$ is a $C_0(E\backslash E\unit)$-algebra it remains to show the map $\phi\mapsto Q_\phi$ is nondegenerate.  
Pick $f\in \Ind{}{}(A,\alpha)$ and let $\epsilon>0$ be given.  Since
$E\cdot u\mapsto \|f(u)\|$ vanishes at infinity, the set $K:=\{E\cdot
u\in E\backslash E\unit: \|f(u)\|\geq\epsilon\}$ is compact. Choose
$\phi\in C_c(E\backslash E\unit)$ such that $0\leq \phi\leq 1$ and
$\phi\equiv 1$ on $K$.  Then
\[
\|f(u)-Q_\phi(f)(u)\|\leq \begin{cases} 0 &\text{if}\quad E\cdot u\in K\\
\|f(u)\| &\text{if}\quad E\cdot u \notin K\end{cases}\quad<\epsilon.
\]
Hence the map $\phi\mapsto Q_\phi $ is nondegenerate  and so
$\Ind{}{}(\A,\alpha)$ is a $C_0(E\backslash E\unit)$-algebra. 

For item~\eqref{R} consider the map $f\mapsto f|_{E\cdot u}$.  Routine computations show
this defines a $*$-homomorphism
$\tilde{R}:\Ind{}{}(A,\alpha)\to\Ind{}{}(\Gamma_0(E\cdot u,
\A),\alpha)$. If $f\in\Ind{}{}(A,\alpha)$ and $\phi\in C_{E\cdot
  u,0}(E\backslash E\unit)$, then $Q_\phi(f)|_{E\cdot u}=0$ so 
\begin{equation}
\label{eq:2}
I_{E\cdot u}:=C_{E\cdot u,0}(E\backslash E\unit)\cdot
\Ind{}{}(A,\alpha)\subset \ker(\tilde{R}).
\end{equation}  
Thus $\tilde{R}$ induces
a $*$-homomorphism  $R$ from $\Ind{}{}(A,\alpha)(E\cdot u)$  into
$\Ind{}{}(\Gamma_0(E\cdot u, \A),\alpha)$.  
We show $R$ is injective by proving $\ker(\tilde{R})=I_{E\cdot u}$.
By \eqref{eq:2}, 
it remains to show $\ker(\tilde{R})\subset I_{E\cdot u}$.  

Let $\epsilon>0$ be given.  Suppose
that $f\in \ker(\tilde{R})$.  Then $f|_{E\cdot u}\equiv 0$.   
Let $K=\{E\cdot u\in E\backslash E\unit: \|f(u)\|\geq \epsilon\}$ and
let $U$ be the complement of $K$ in $E\backslash E\unit$.   By assumption
$E\cdot u\in U$, and $U$ is open since $K$ is a compact subset of the
Hausdorff space $E\backslash E\unit$.  Pick  functions $\phi,\psi\in
C_c(E\backslash E\unit)$ such that $0\leq \phi,\psi\leq 1$,
$\supp(\phi)\subset U$, $\phi(E\cdot u)=1$, and $\psi|_K\equiv 1$.
Then $(1-\phi)\psi\in C_{E\cdot u,0}(E\backslash E\unit)$ and 
\[
\|f(v)-Q_{(1-\phi)\psi}(f)(v)\|\leq \begin{cases} 0& \text{if}\quad
  E\cdot v\in K \\ \|f(v)\| &\text{if}\quad E\cdot v\notin
  K\end{cases}\quad <\epsilon.
\]
Therefore $f\in I_{E\cdot u}$, giving $I_{E\cdot u}=\ker(\tilde{R})$, and thus
$R:\Ind{}{}(A,\alpha)(E\cdot u)\to \Ind{}{}(\Gamma_0(E\cdot u, \A),
\alpha)$ is injective. 

To show $R$ is an isomorphism, it remains to show it is surjective.
Pick a function $F\in \Ind{}{}(\Gamma_0(E\cdot u, \A),\alpha)$ and consider
$F(u)\in A(u)$. Let $\epsilon>0$ be given.  By Lemma~\ref{claim:
  fibre dense}, there exists $f\in \Ind{}{}(A,\alpha)$  such that
$\|f(u)-F(u)\|<\epsilon.$  For $v\in E\cdot u$ there exists $\gamma\in
E_u$ such that $v=r(\gamma)$.   Thus
$\|f(v)-F(v)\|=\|f(r(\gamma))-F(r(\gamma))\|=\|\alpha_\gamma(f(u))-\alpha_\gamma(F(u))\|<\epsilon.$
Therefore the image of $R$ is dense in $ \Ind{}{}(\Gamma_0(E\cdot u,
\A),\alpha)$ and thus $R$ is surjective.   

For item~\eqref{evaluation}, by Lemma~\ref{claim: fibre dense}, the map
$\varepsilon_u:\Ind{}{}(\Gamma_0(E\cdot u,
\A),\alpha)\to A(u)$ given by $f\mapsto f(u)$ has a dense range. It is
straightforward to show that $\varepsilon_u$ is a $*$-homomorphism.  It is isometric since $\|f\|_\infty=\|f(u)\|$ for all $f\in\Ind{}{}(\Gamma_0(E\cdot u,
\A),\alpha)$, and thus $\varepsilon_u$ is an isomorphism as desired. 
\end{proof}

\begin{rmk}
To differentiate between the $C_0(E\backslash E\unit)$-algebra
$\Ind{E}{E\unit}(A,\alpha)$ and the corresponding upper semicontinuous
$C^*$-bundle we will denote the upper semicontinuous $C^*$-bundle by
$\Ind{E}{E\unit}(\A,\alpha)$.
\end{rmk}

Let $E$ be a principal and proper groupoid and suppose $(A,\alpha)$ is
an $E$-dynamical system.  Then \cite{mep09} defines the generalized
fixed point algebra $\Fix(A,\alpha)$ to be the closed span of
$\lambda ( a^* b)$ in $M(A)$ for $a,b\in
\Gamma_c(E\unit,\A)$.\footnote{$\Fix(A,\alpha)$ is denoted $A^\alpha$
  in \cite{mep09}.} By Lemma~\ref{lem: def lambda},
$\Fix(A,\alpha)\subset \Ind{}{}(A,\alpha)$.  In the following
proposition we show that $\Fix(A,\alpha)= \Ind{}{}(A,\alpha)$.

\begin{prop}
\label{prop: gen fix ind}
Let $E$ be a principal and proper groupoid and $(A,\alpha)$ an
$E$-dynamical system.  Then the generalized fixed point algebra
$\Fix(A,\alpha)$ is equal to $\Ind{E}{E\unit}(A, \alpha)$. 
\end{prop}

\begin{proof}
The proof of \cite[Proposition~4.4]{mep09} shows that $\Fix(A,\alpha)$ is a
$C_0(E\backslash E\unit)$-subalgebra of $\Ind{}{}(A,\alpha)$.  By
Proposition~\ref{prop: when dense}, to show $\Fix(A,\alpha)$ is all of
$\Ind{}{}(A,\alpha)$ it suffices to prove $\Fix(A,\alpha)(E\cdot u)$
is dense in $\Ind{}{}(A,\alpha)(E\cdot u)$ for every $u\in E\unit$.   

Let $g\in \Ind{}{}(A,\alpha)$, $u\in E\unit$, and $\epsilon>0$ be
given.  Pick an approximate unit $\{e_i\}$ for $A(u)$ and $i_0$
large enough so that $\|e_{i_0}g(u)-g(u)\|<\epsilon/2$.  Choose $f\in
\Gamma_c(E\unit,\A)$ such that $f(u)=e_{i_0}$.  Then $\|f(u)
g(u)-g(u)\|<\epsilon/2$.  By Lemma~\ref{claim: fibre dense} there
exists $\psi\in C_c(E\unit)$ such that $\|\lambda(\psi\cdot
(fg))(u)-(fg)(u)\|<\epsilon/2$.  Thus $\|\lambda(\psi\cdot
(fg))(u)-g(u)\|<\epsilon$.  Furthermore, since $f$ and $u\mapsto \psi(u)g(u)$
 are both
in $\Gamma_c(E\unit, \A)$, $\lambda(\psi\cdot (fg))\in
\Fix(A,\alpha)$.
\end{proof}

\begin{rmk}
We identify $\Fix(A,\alpha)$ with $\Ind{E}{E\unit}(A,\alpha)$.
As with $\Ind{E}{E\unit}(A,\alpha)$ we
 denote the upper semicontinuous $C^*$-bundle corresponding to
$\Fix(A,\alpha)$ by $\Fix(\A,\alpha)$. 
\end{rmk}

We next study the representations and pure states of
$\Ind{}{}(A,\alpha)$ by showing they come from representations and
pure states of $A$.  We will use this analysis in
Proposition~\ref{prop: sx fix=id}. 

Let $B$ be a $C^*$-algebra and denote by $P(B)$ the pure states on
$B$, $\widehat{B}$ the irreducible representations on $B$, and
$\Lambda_B: P(B)\to \widehat{B}$ the map given by the GNS
construction.   That is, for $\tau\in P(B)$ there exists a unit vector
$h$ such that 
 $\tau(a)=(\Lambda_B(\tau)(a)h, h)$.  If $B$ is a $C_0(T)$-algebra, let
$\sigma_B$ be the associated map from $\widehat{B}$ to $T$.  For $t\in
T$ let $q_t: B\to B(t)$ be the quotient map.  Suppose $\pi\in
\widehat{B}$ and $t=\sigma_B(\pi)$.  By \cite[Proposition~C.5]{TFB2}
there exists a $\pi_{t}\in \widehat{B(t)}$ such that $\pi=\pi_t\circ
q_t$. 
Thus $\pi(a)$ depends only on the class of $a$ in $A(t)$. Similarly,
for $\tau\in P(B)$ and $t=\sigma_B(\Lambda_B(\tau))$,
\[
\tau(a)=(\Lambda_B(\tau)(a)h,h)=(\Lambda_B(\tau)_{t}(q_{t}(a))h,h),
\]
so that $\tau$ depends only on the class of $a$ in $A(t)$.  It
follows there exists $\tau_{t}$ such that $\tau(a)=\tau_{t}(q_{t}(a))$.

Let $E$ be a principal and proper groupoid and  $(A, \alpha)$ an
$E$-dynamical system.  Recall that $A$ must be a
$C_0(E\unit)$-algebra.   Suppose $\pi\in \widehat{A}$ and let
$u=\sigma_A(\pi)$. By \cite[Proposition~C.5]{TFB2} there exists a
$\pi_{u}\in \widehat{A(u)}$ such that $\pi=\pi_u\circ q_u$. For
$\gamma\in E_u$,  
\[
\gamma\cdot \pi(a):=\pi_u\circ\alpha_\gamma\inv\circ q_{r(\gamma)} (a)
\] 
defines a continuous action of $E$ on $\widehat{A}$ \cite[Proposition
1.1]{goe:mackey2}. 

Recall from Proposition~\ref{lem:  ind fix} that $\Ind{}{}(A,\alpha)$ is a
$C_0(E\backslash E\unit)$-algebra and the evaluation map
$\varepsilon_u: \Ind{}{}(A,\alpha)\to A(u)$ induces an isomorphism of
$\Ind{}{}(A,\alpha)(E\cdot u)$ with $A(u)$. We can define a representation $M(\pi)$ of $\Ind{}{}(A,\alpha)$ by 
\begin{equation}
\label{eq:M}
M(\pi)(f)= \pi_{\sigma_A(\pi)}(\varepsilon_{\sigma_A(\pi)}(f))
\end{equation}
for $f\in \Ind{}{}(A,\alpha)$.  Now $M(\pi)(\Ind{}{}(A,\alpha))=\pi_{\sigma_A(\pi)}(A({\sigma_A(\pi)}))=\pi(A)$.  Thus, since $\pi$ is irreducible, so is $M(\pi)$. 
Similarly for $\tau \in P(A)$, let $v=\sigma_A\circ
\Lambda_A(\tau)$. We can define a  state on $\Ind{}{}(A,\alpha)$
by $N(\tau)=\tau_{v}(\varepsilon_v(f))$.  It is straightforward to
show that 
\begin{equation}
\label{eq:3}
\Lambda_{\Ind{}{}(A,\alpha)}(N(\tau))=M(\Lambda_A(\tau)),
\end{equation}
which implies that $N(\tau)$ is pure.

Our goal is to show that $M$ defines a continuous, open bijection from $E\backslash \widehat{A}$ to $(\Ind{}{}(A,\alpha))\sidehat$. For this we first show $N$ is continuous.

\begin{lem}
\label{lem: N cont}
 $N$ is continuous for the weak-$*$ topology.
 \end{lem}	
 
\begin{proof}
 Suppose $\tau_i\to \tau$.  By definition $\tau_i\to \tau$ if and only
 if $\tau_i(a)\to \tau(a)$ for all $a\in A$.  Pick $f\in
 \Ind{}{}(A,\alpha)$.  Define $v_i=\sigma_A\circ \Lambda_A(\tau_i)$
 and $v=\sigma_A\circ \Lambda_A(\tau)$. Since $\sigma_A$ and
 $\Lambda_A$ are continuous $v_i\to v$.  Thus there exists a compact
 set $K\subset E\unit$ such that $\{v_i, v\}\subset K$.  Pick $\phi\in
 C_c(E\unit)$ such that $\phi|_K\equiv 1$, then $\phi\cdot f\in
 A$. Since $\phi\cdot f(v_i)=f(v_i)$,  we have
 $(\tau_i)_{v_i}(\varepsilon_{v_i}(\phi\cdot
 f))=(\tau_i)_{v_i}(\varepsilon_{v_i}(f))$.  Hence 
\begin{align*}
N(\tau_i)(f)&=(\tau_i)_{v_i}(\varepsilon_{v_i}(f))=(\tau_i)_{v_i}(\varepsilon_{v_i}(\phi\cdot
f))
=\tau_i(\phi\cdot f) \\ &\to \tau(\phi\cdot
f)=(\tau)_{v}(\varepsilon_{v}(f))=N(\tau)(f),
\end{align*}
so $N$ is continuous.
\end{proof}

\begin{lem} 
 \label{lem: rep cor} 
  For each $\pi\in \widehat{A}$ define 
  $M(\pi): \Ind{E}{E\unit}(A,\alpha)\to B(\h_\pi)$ as in \eqref{eq:M}.
Then $M$ induces a homeomorphism of $E\backslash \widehat{A}$ to
$(\Ind{E}{E\unit}(A,\alpha))\sidehat$. 
\end{lem}

\begin{proof}
We showed above that $M(\pi)$ is irreducible for $\pi\in \widehat{A}.$
Note that $M$ descends to a well defined map of  $E\backslash
\widehat{A}$ to $(\Ind{E}{E\unit}(A,\alpha))\sidehat$ since  
\begin{align*}
M(\gamma\cdot\pi)(f)&=(\gamma\cdot \pi)_{r(\gamma)}(
f(r(\gamma)))=\pi_{s(\gamma)}
(\alpha_{\gamma\inv}(f(r(\gamma))))\\
&=\pi_{s(\gamma)}(f(s(\gamma)))=M(\pi)(f). 
\end{align*}

We now show that $M$ induces a bijection of $E\backslash \widehat{A}$
and $(\Ind{}{}(A,\alpha))\sidehat$. To see $M$ is surjective, suppose
$\rho$ is an irreducible representation on $\Ind{}{}(A,\alpha)$. We
can factor $\rho$ to an irreducible representation of some fibre
$\Ind{}{}(A,\alpha)(E\cdot v)$, transport it to $A(v)$ via the
isomorphism of Proposition~\ref{lem:  ind fix}, and then lift to a
representation $\pi$ of $A$.  Tracing through definitions shows that
$M(\pi) = \rho$.   To see $M$ is injective, suppose $M(\pi)$ is unitarily equivalent to 
$M(\rho)$.  Let
$u=\sigma_A(\pi)$ and $v=\sigma_A(\rho)$. By definition,
$M(\pi)=\pi_u\circ \varepsilon_u$ and $ M(\rho)=\rho_v\circ
\varepsilon_v$.  By \cite[Proposition C.5]{TFB2}, $M(\pi)$ and
$M(\rho)$ both factor to representations of some fibre
$\Ind{}{}(A,\alpha)(E\cdot w)$.  However, since the quotient map onto
the fibre is given
by restriction, the only way these statements can be
compatible is if $u,v\in E\cdot w$. Thus there exists
a $\gamma\in E $ such that $u=s(\gamma)$ and $v=r(\gamma)$; that is
$\pi_u\in \widehat{A(s(\gamma))}$ and $\rho_v \in
\widehat{A(r(\gamma))}$. Since $M(\pi)$ is equivalent to $M(\rho)$,
there exists a unitary $U$ such that  
\[
\pi_{s(\gamma)}\circ
\varepsilon_{s(\gamma)}(f)=U\rho_{r(\gamma)}\circ\varepsilon_{r(\gamma)}(f)
U^*=U\rho_{r(\gamma)}\circ f(r(\gamma))
U^*=U\rho_{r(\gamma)}(\alpha_{\gamma}(f(s(\gamma))))U^* 
\]
which implies that $\pi$ is unitarily equivalent to 
$\gamma\inv \cdot \rho$.  Thus $M$ induces a bijection. 

Since $N$ is continuous by Lemma~\ref{lem: N cont} and
$\Lambda_{\Ind{}{}(A,\alpha)}(N(\tau))=M(\Lambda_A(\tau))$ by \eqref{eq:3}
we get that $M$ is continuous as well.
Finally we show that $M$ is open. Suppose $M(\pi_i)\to M(\pi)$ in
$(\Ind{E}{E\unit}(A,\alpha))\sidehat$.  Let $u_i = \sigma_A(\pi_i)$
and $u = \sigma_A(\pi)$.  
Since $M(\pi_i)\to M(\pi)$, 
a straightforward argument shows that $E\cdot u_i \to E\cdot u$.    
The map from $E\unit$ onto $E\backslash E\unit$ is
open so we may pass to a subnet, relabel, and chose $\gamma_i$ such that
$r(\gamma_i)=\gamma_i \cdot u_i \to u$.  To prove $\gamma_i\cdot\pi_i \to \pi$ it
suffices to show that if $J$
is an ideal in $A$ such that $J \not\subset \ker\pi$ then eventually
$J\not\subset \ker\gamma_i\cdot \pi_i$ \cite[Corollary~A.28]{tfb}.  Choose $a\in J$ such that
$\pi(a) = \pi_u(a(u)) \ne 0$.  Let $\epsilon = \|\pi(a)\|/4$ and
use Lemma~\ref{claim: fibre dense} to
find $f\in \Ind{}{}(A,\alpha)$ such that $\|f(u) - a(u)\| <
\epsilon$.  Since the norm on $\A$ is upper semicontinuous the
set $\{ b\in \A : \|b\| < \epsilon \}$ is open.  Both $f$
and $a$ are continuous as functions on $E\unit$ so 
$f(r(\gamma_i))-a(r(\gamma_i))\to f(u)-a(u)$ and this
implies that, eventually,  
$\|f(r(\gamma_i))-a(\gamma_i \cdot u_i)\|<\epsilon$.  Next,
observe that by construction 
\(
M(\pi)(f) = \pi_u(f(u))
\)
and that 
\[
|\|M(\pi)(f)\| - \|\pi(a)\|| \leq \|\pi_u(f(u)) - \pi_u(a(u))\|
\leq \|f(u)-a(u)\| < \epsilon.
\]
Since $\epsilon = \|\pi(a)\|/4$ this implies that 
\(
\|M(\pi)(f)\| > \frac{\|\pi(a)\|}{2}.
\)
We know from \cite[Lemma~A.30]{tfb} that the map $\rho \to \rho(f)$ is
lower semicontinuous on $(\Ind{E\unit}{E}(A,\alpha))\sidehat$ so 
the set $\{ \rho : \|\rho(f)\| > \|\pi(a)\|/2\}$ is open.  Since
we assumed $M(\gamma_i\cdot \pi_i)\to M(\pi)$, this implies that for
large $i$ 
\[
\|M(\gamma_i\cdot \pi_i)(f)\| = \|(\gamma_i\cdot \pi_i)_{\gamma_i
  \cdot u_i}(f(\gamma_i \cdot u_i))\| > \|\pi(a)\|/2.
\]
However, we also eventually have
\begin{align*}
|\|M(\gamma_i\cdot \pi_i)(f)\| - \|\gamma_i\cdot \pi_i(a)\|| 
&\leq
\|(\gamma_i\cdot \pi_i)_{r(\gamma_i)}(f(r(\gamma_i))) -
(\gamma_i\cdot \pi_i)_{r(\gamma_i)}(a(r(\gamma_i)))\| \\
&\leq \|f(r(\gamma_i))-a(r(\gamma_i ))\| < \epsilon
\end{align*}
Therefore, for large enough $i$,
\(
\|\gamma_i\cdot \pi_i(a)\| > \frac{\|\pi(a)\|}{4}
\)
and in particular $a\not\in \ker \gamma_i\cdot \pi_i$.  This suffices to
show that eventually $J\not\subset \ker\gamma_i\cdot \pi_i$, and thus
$M$ is open.  
\end{proof}

\begin{cor}
\label{cor: crossed spect}
Let $E$ be a principal and proper groupoid and $(A,\alpha)$ be an
$E$-dynamical system.  Then $(A\rtimes_\alpha E)\sidehat\cong
E\backslash \widehat{A}.$ 
\end{cor}

\begin{proof}
Since $E$ is principal and proper \cite[Theorem~5.2]{mep09} shows that
$(A,\alpha)$ is a ``saturated'' $E$-dynamical system.  That is,
$\Fix(A,\alpha)$ is Morita equivalent to $A\rtimes_{\alpha, r} E$.
Now $E$ is proper and hence amenable, so
$A\rtimes_{\alpha, r} E\cong A\rtimes_{\alpha} E$
\cite[Corollary~2.1.7, Proposition~6.1.8]{A-DR00}. Thus $(A\rtimes_\alpha
E)\sidehat\cong (A\rtimes_{\alpha, r} E)\sidehat\cong
\Fix(A,\alpha)\sidehat$.  But by Proposition~\ref{prop: gen fix ind},
$\Fix(A,\alpha)\cong \Ind{}{}(A,\alpha)$ and by Lemma~\ref{lem: rep
  cor},  $\Ind{}{}(A,\alpha)\sidehat \cong E\backslash \widehat{A}$; therefore $(A\rtimes_\alpha
E)\sidehat\cong E\backslash \widehat{A}$. 
\end{proof}

\begin{rmk} 
Note that $E$ principal implies the isotropy subgroupoid of $E$
is trivial; further $E$ proper implies $E\backslash E\unit$ is Hausdorff
and therefore ``regular'' in the sense of \cite{goe:mackey2}.  Thus
Corollary~\ref{cor: crossed spect} is a special case of
\cite[Theorem~2.22]{goe:mackey2}. 
\end{rmk}

%%%%%%%%%%%%%%%%%%%%%%%%%%%%%%%%%%%%%%%%%%%%%%%%%%%%%%%%%%%%%%%%%%%%%%%%%%%%%%%%%%%%%%%%%%%%%%%%%%%%%%%%%%%%%%%%%%%%%%%%%%%%%%%%%%%%%%%%%%%%%%%%%%%%%%%

\section{The  Brauer semigroup}
\label{sec: brauer}

Let $E$ be a second countable locally compact Hausdorff groupoid. As
in \cite{aHRW00} and \cite{KMRW98} we want to define a commutative
binary operation on classes of $E$-dynamical systems. 
In those papers the binary operation is induced by a
balanced tensor product.  In \cite{KMRW98} they consider only those
$E$-dynamical systems $(A,\alpha)$ with $A$ continuous trace, and hence
nuclear.  Thus \cite{KMRW98} does not need to specify a particular
tensor product.  Since we are considering all (separable)
$E$-dynamical systems we must make a choice and so we follow
\cite{aHRW00} and use the maximal balanced tensor product introduced
in \cite{Blan96} for compact spaces. The results from \cite{Blan96}
can be easily extended to arbitrary locally compact Hausdorff spaces
so we cite them without further comment.

Let $T$ be a second countable locally compact Hausdorff space, and $A$
and $B$ be $C_0(T)$-algebras. Consider the ideal $J_T$ of
$A\otimes_{\max} B$ generated by  
\[
\{(\phi\cdot a)\otimes b-a\otimes(\phi\cdot b): a\in A,~ b\in
B,~\phi\in C_0(T)\}. 
\]
We define the $C_0(T)$-balanced tensor product by
\[
A\otimes_T B:= A\otimes_{\max} B/J_T.
\]
The map $\phi\cdot (a\otimes b):=(\phi\cdot a\otimes b)$ defines a
$C_0(T)$-structure on $A\otimes_T B$.  Furthermore, if $\A\otimes_T\B$
is the upper semicontinuous $C^*$-bundle over $T$ constructed from
$A\otimes_T B$, then we have  $(A\otimes_T B)(t)=A(t)\otimes_{\max}
B(t)$ \cite[Lemma~2.4]{EW98}.  Note that if $X$ is a locally compact
Hausdorff space and $q:X\to T$ is a continuous open surjection then
$C_0(X)$ is a $C_0(T)$-algebra and 
$C_0(X)\otimes_T A\cong \Gamma_0(X, q^*\A)$ \cite[Proposition~1.3]{RW85}. 

Let $(A,\alpha)$ and $(B,\beta)$ be $E$-dynamical systems.
Since $\alpha_\gamma:A(s(\gamma))\to A(r(\gamma))$ and $\beta_\gamma:
B(s(\gamma))\to B(r(\gamma))$ are isomorphisms, we can use
\cite[Lemma~B.31]{tfb} to show that the map $\alpha_\gamma\otimes
\beta_\gamma: A(s(\gamma))\otimes_{\max} B(s(\gamma))\to
A(r(\gamma))\otimes_{\max} B(r(\gamma))$ characterized by
$\alpha_\gamma\otimes \beta_\gamma(a\otimes b)=\alpha_\gamma(a)\otimes
\beta_\gamma(b)$ is an isomorphism.  The argument on page 919 of
\cite{KMRW98} shows that the collection $\{\alpha_\gamma\otimes
\beta_\gamma\}$ defines a continuous $E$-action on $A\otimes_{E\unit}
B$.   Lemma~2.4 of \cite{aHRW00} implies that the $C_0(E\unit)$-balanced
tensor product is an
associative, commutative, binary operation on the set of $E$-dynamical
systems.  

\begin{defn}[{\cite[Definition~3.1]{KMRW98}}]
\label{def equi morita}
Two $E$-dynamical systems $(A,\alpha)$ and $(B,\beta)$ are
\emph{equivariantly Morita equivalent}  if there is an
$\A-\B$-imprimitivity bimodule bundle $\Zb$ which admits an action $V$
of $E$ by isomorphisms such that  
\[
\linner{A(r(\gamma))}{V_\gamma(\xi)}{V_\gamma(\zeta)}=\alpha_\gamma(\linner{A(r(\gamma))}{\xi}{\zeta})\text{~~and~~}
\rinner{B(r(\gamma))}{V_\gamma(\xi)}{V_\gamma(\zeta)}=\beta_\gamma(\rinner{A(r(\gamma))}{\xi}{\zeta}). 
\]
In this case we will write $(A,\alpha)\sim_{(Z,V)}(B,\beta)$.
\end{defn}

\begin{rmk}
If $(\Zb,V)$ is an $(A,\alpha)-(B,\beta)$ equivariant imprimitivity
bimodule bundle then a computation shows that $V_\gamma(a\cdot z)=\alpha_\gamma(a)
\cdot V_\gamma(z)$ and $V_\gamma(z\cdot b) = V_\gamma(z)\cdot
\beta_\gamma(b)$.  Furthermore, it follows from \cite[Section
9.1]{MW08} that the corresponding crossed products $A\rtimes_\alpha E$
and $B\rtimes_\beta E$ are Morita equivalent.    
\end{rmk}

The proof of the following lemma is similar to the proof of
\cite[Lemma~3.2]{KMRW98} and has been omitted. 

\begin{lem}
\label{lem equiv rel}
Equivariant Morita equivalence is an equivalence relation. For an
$E$-dynamical system $(A,\alpha)$ we denote its equivariant Morita
equivalence class by $[A,\alpha]$. 
\end{lem}

Next we show that the balanced tensor product gives us a
semigroup operation.  

\begin{prop}
\label{prop brauer semi}
Let $[A,\alpha]$ and $[B,\beta]$ be equivariant Morita equivalence
classes of $E$-dynamical systems.  Then 
\begin{equation}
\label{eq bin op}
[A,\alpha][B,\beta]:=[A\otimes_{E\unit}B, \alpha\otimes \beta]
\end{equation}
is a  well defined commutative binary operation with identity $[C_0(E\unit),\lt]$.
\end{prop}

\begin{proof}
Since $\otimes_{E\unit}$ is an associative, commutative, binary
operation on $E$-dynamical systems, it suffices to show that the
multiplication in \eqref{eq bin op} is well defined. Suppose
$(A,\alpha)\sim_{(X,V)} (C,\vartheta)$ and $(B, \beta)\sim_{(Y,W)}
(D,\delta)$.  Then as in the proof of \cite[Proposition~3.6]{KMRW98}
we can define an imprimitivity bimodule bundle $\Zb$ with fibres given
by the external tensor product $X(u)\otimes Y(u)$ under the inner
products characterized by 
\begin{align*}
\linner{A(u)\otimes B(u)}{x\otimes y}{x'\otimes y'}&:=\linner{A(u)}{x}{x'}\otimes \linner{B(u)}{y}{y'}\quad\text{and}\quad\\
\rinner{C(u)\otimes D(u)}{x\otimes y}{x'\otimes y'}&:=\rinner{C(u)}{x}{x'}\otimes \rinner{D(u)}{y}{y'}.
\end{align*}
The topology on $\Zb$ is characterized by the condition that $u\mapsto f(u)\otimes g(u)$ is continuous for all $f\in X$, $g\in Y$.
The continuity of the left and right actions follows as in the proof of \cite[Proposition~3.6]{KMRW98}.  It remains to show continuity of the inner products. 

By symmetry it suffices to show the $A\otimes_{E\unit} B$ valued inner
product is continuous. Let $z_i\to z$ and $w_i\to w$ in $\Zb$  with
$p_\Zb(w_i)=p_\Zb(z_i)=u_i$ and $p_\Zb(w)=p_\Zb(z)=u$. Let
$\epsilon>0$.  Pick finite subsets $J, K\subset X(u)\times Y(u)$ such
that $\|z-\sum_{(x,y)\in J} x\otimes y\|< \epsilon$ and
$\|w-\sum_{(x',y')\in K} x'\otimes y'\|< \epsilon$.  Let $\pi_i$ be
the projection onto the $i$-th factor.  For each $x\in \pi_1(J\cup
K)$ pick $f_x\in X$ such that $f_x(u)=x$ and similarly for each $y\in
\pi_2(J\cup K)$ pick $g_y\in Y$ such that $g_y(u)=y$.  By the
continuity of $f_x$, $g_y$, and the inner products we
have $\linner{A(u_i)}{f_x(u_i)}{f_{x'}(u_i)}\to
\linner{A(u)}{f_x(u)}{f_{x'}(u)}$ and
$\linner{B(u_i)}{g_y(u_i)}{g_{y'}(u_i)}\to
\linner{B(u)}{g_y(u)}{g_{y'}(u)}$.  By the definition of the topology
on $\A\otimes_{E\unit} \B$ we then get  
\begin{align*}
&\linner{A\otimes B (u_i)}{f_x\otimes g_y(u_i)}{f_{x'}\otimes
  g_{y'}(u_i)}=\linner{A(u_i)}{f_x(u_i)}{f_{x'}(u_i)}\otimes
\linner{B(u_i)}{g_y(u_i)}{g_{y'}(u_i)}\to \\ 
 &\hspace{1.5 cm} \linner{A(u)}{f_x(u)}{f_{x'}(u)}\otimes
 \linner{B(u)}{g_y(u)}{g_{y'}(u)}=\linner{A\otimes B (u)}{f_x\otimes
   g_y(u)}{f_{x'}\otimes g_{y'}(u)} 
\end{align*}
Thus 
\begin{align*}&\linner{A\otimes B (u_i)}{\sum_{(x,y)\in J}f_x\otimes
    g_y(u_i)}{\sum_{(x',y')\in K}f_{x'}\otimes g_{y'}(u_i)}\to \\ 
&\hspace{2 cm}\linner{A\otimes B (u)}{\sum_{(x,y)\in J}f_x\otimes
  g_y(u)}{\sum_{(x',y')\in K}f_{x'}\otimes g_{y'}(u)}. 
\end{align*} By the definition of the norms on $Z(u_i)$ and $ Z(u)$,
and the Cauchy-Schwartz inequality, we have  
\[
\left\|\linner{A\otimes B (u_i)}{\sum_{(x,y)\in J}f_x\otimes
    g_y(u_i)}{\sum_{(x',y')\in K}f_{x'}\otimes
    g_{y'}(u_i)}-\linner{A\otimes B(u_i)}{z_i}{w_i}\right\|<2\epsilon(\|z\|+\epsilon).
\]
Proposition~\ref{Proposition c.20} now gives that $\linner{A\otimes B
  (u_i)}{z_i}{w_i}\to \linner{A\otimes B (u)}{z}{w}$.  Thus the
$\A\otimes_{E\unit} \B$-valued inner product is continuous as desired,
and the multiplication in \eqref{eq bin op} is well-defined.

Finally, showing $[C_0(E\unit), \lt]$ acts as an identity for this multiplication follows as in the proof of \cite[Proposition~3.6]{KMRW98}.
\end{proof}

\begin{defn}
\label{def brauer semi}
Let $E$ be a locally compact Hausdorff groupoid.  The
\emph{Brauer semigroup}, denoted $S(E)$, to be the abelian monoid
of equivariant Morita equivalence classes of $E$-dynamical systems
with the operation given in Proposition~\ref{prop brauer semi}.  The
\emph{Brauer group} is the set $\Br(E)=\{[A,\alpha]\in S(E): A \text{~ has continuous trace}\}$.
\end{defn}

\begin{rmk}
\label{rmk brauer group}
Lemma~6.6 of \cite{aHRW00} states that $A\otimes_{E\unit} B$ has
continuous trace with spectrum $E\unit$ if and only if $A$ and $B$
have continuous trace with spectrum $E\unit$.  Thus if  $[A,\alpha]$
is invertible, then $A$ must have continuous trace with spectrum
$E\unit$ and therefore be an element of the Brauer group
$\Br(E)$. Conversely, if $A$ has continuous trace with spectrum
$E\unit$ then $[A,\alpha]\in\Br(E)$ and  \cite[Theorem~3.7]{KMRW98}
says $[A,\alpha]$ is invertible.  That is, the invertible elements of
$S(E)$ are  precisely those in $\Br(E)$.   
\end{rmk}

%%%%%%%%%%%%%%%%%%%%%%%%%%%%%%%%%%%%%%%%%%%%%%%%%%%%%%%%%%%%%%%%%%%%%%%%%%%%%%%%%%%%%%%%%%%%%%%%%%%%%%%%%%%%%%%%%%%%%%%%%%%%%%%%%%%%%%%%%%%%%%%%%%%%%%%

\section{Main Theorem}
\label{sec:main-theorem}

Throughout suppose $G$ and $H$ are second countable locally compact
Hausdorff groupoids with Haar systems $\{\lambda_G^u\}_{u\in
  G\unit}$ and $\{\lambda_H^u\}_{u\in H\unit}$, and let $X$
be a $(G,H)$-equivalence \cite[Definition~2.1]{MRW87}. In particular
$G$ and $H$ act freely and properly on the left and right of $X$,
respectively, and $r_X: X\to H\unit$  and $s_X: X\to G\unit$ induce
homeomorphisms $G\backslash X\cong H\unit$ and $X/ H\cong G\unit$,
respectively.  We can define a transformation groupoid  $G\ltimes
X\rtimes H:=\{(\gamma,x,\eta)\in G\times X\times H:
r_G(\gamma)=r_X(x), s_X(x)=r_H(\eta)\}$ associated to this equivalence
whose topology is given by the relative topology and whose operations
are defined by
\begin{align*}
r(\gamma,x,\eta)&=x & s(\gamma,x,\eta) &=\gamma\inv x\eta\\
(\gamma,x,\eta)(\xi, \gamma\inv x\eta, \zeta)&=(\gamma\xi,
x,\eta\zeta) & (\gamma,x,\eta)\inv &=(\gamma\inv, \gamma\inv
x\eta,\eta\inv). 
\end{align*}
The transformation groupoids $G\ltimes X$ and $X\rtimes H$ embed
naturally in $G\ltimes X\rtimes H$ by $(\gamma,x)\mapsto (\gamma,x,
s_X(x))$ and $(x,\eta)\mapsto (r_X(x),x,\eta)$, respectively.  We
identify $G\ltimes X$ and $X\rtimes H$ with their image under these
embeddings.   

Suppose $(A, \omega)$ is a $G\ltimes X\rtimes H$-dynamical system.
Then $\omega^G:=\omega|_{G\ltimes X}$ and $\omega^H:=\omega|_{X\rtimes
  H}$ are continuous actions of $G\ltimes X$ and $X\rtimes H$ on $A$,
respectively.  Furthermore, since both sides are equal to $\omega_{(\gamma,x,
  \eta)}$, we have  
\begin{equation}
\label{eq action commute}
\omega^G_{(\gamma,x)}\circ \omega^H_{(\gamma\inv x,
  \eta)}=\omega^H_{(x,\eta)}\circ\omega^G_{(\gamma,x\eta)}. 
\end{equation}
The next proposition shows that every action of $G\ltimes X\rtimes H$
arises in this way. 

\begin{prop}
\label{prop com act}
Let $A$ be a $C_0(X)$-algebra and $X$ a $(G,H)$-equivalence.  Then $A$
admits a $G\ltimes X\rtimes H$ action $\omega$ if and only if there
exists actions $\omega^G$ and $\omega^H$ of $G\ltimes X$ and $X\rtimes
H$ on $A$ that satisfy equation~\eqref{eq action commute}.
\end{prop}

\begin{proof}
The only if part was shown above.  Suppose that $A$ admits actions
$\omega^G$ and $\omega^H$ of $G\ltimes X$ and $X\rtimes H$ that satisfy
equation~\eqref{eq action commute}.  Define
$\omega_{(\gamma,x,\eta)}:=\omega^G_{(\gamma,x)}\circ\omega^H_{(\gamma\inv
  x,\eta)}$.  Using \eqref{eq action commute} one can show that
$\omega$ is well-defined and respects the groupoid multiplication.  It
is an isomorphism of the fibres  and  continuous since $\omega^G$
and $\omega^H$ are. 
\end{proof}

Let $(A,\alpha)$ be an $H$-dynamical system.  Then $s_X^*A$ is a $C_0(X)$-algebra and we can define maps 
\[
s_X^*\alpha_{(\gamma,x,\eta)}: s_X^*A(\gamma\inv x\eta)\to s_X^*A(x)\quad\text{by}\quad (\gamma\inv x\eta, a) \mapsto (x,\alpha_\eta(a)).
\]
We show in Proposition~\ref{prop pull back cont} that $s_X^*\alpha$
defines a continuous action of $G\ltimes X\rtimes H$ on $s_X^*A$. Similar statements hold for $G$-dynamical systems and $r_X$.
Our goal is to prove the following theorem.

\begin{thm}
\label{thm iso semigroup} Suppose $G$ and $H$ are second countable locally compact Hausdorff
groupoids with Haar systems and $X$ is a $(G,H)$-equivalence. Then the following
statements hold.
\begin{enumerate}
\item\label{H part} The map $[B,\beta]\mapsto [s_X^*B, s_X^*\beta]$
  defines an isomorphism $\upsilon^{X, H}:S(H)\to S(G\ltimes X\rtimes
  H)$   such that $s_X^*B\rtimes_{s_X^*\beta}( G\ltimes X\rtimes H)$
  is Morita equivalent to $B\rtimes_\beta H$.  
\item \label{G part} The map $[A,\alpha]\mapsto [r_X^*A, r_X^*\alpha]$
  defines an isomorphism $\upsilon^{X, G}:S(G)\to S(G\ltimes X\rtimes
  H)$   such that $r_X^*A\rtimes_{r_X^*\alpha}( G\ltimes X\rtimes H)$
  is Morita equivalent to $A\rtimes_\alpha G$.   
\item \label{sym iso} The map $\upsilon^X= (\upsilon^{X,G})\inv \circ
  \upsilon^{X, H}$ defines an isomorphism from $S(H)$ to $S(G)$ such
  that if $\upsilon^X([B,\beta])=[A,\alpha]$ then $A\rtimes_{\alpha}
  G$ is Morita equivalent to $B\rtimes_{\beta} H$. 
\end{enumerate}
\end{thm}

To prove Theorem~\ref{thm iso semigroup}  it suffices to to prove item~\eqref{H part} since the others will follow by symmetry.  To do this we will first analyze $\upsilon^{X,H}$ and then define an inverse.  

%%%%%%%%%%%%%%%%%%%%%%%%%%%%%%%%%%%%%%%%%%%%%%%%%%%%%%%%%%%%%%%%%%%%%%%%%%%%%%%%%%%%%%%%%%%%%%%%%%%%%%%%%%%%%%%%%%%%%%%%%%%%%%%%%%%%%%%%%%%%%%%%%%%%%%%

\subsection{The map $\upsilon^{X,H}$ and its properties}
\label{section nu}

Let $\Zb$ be an upper semicontinuous Banach bundle over $H\unit$, $Z=\Gamma_0(H\unit, \Zb)$, and consider $s_X^*\Zb$.  Define $s_X^*Z:=\Gamma_0(X,s_X^*\Zb)$.  

\begin{prop}
\label{prop pull back cont}
Let $\Zb$ be an upper semicontinuous Banach bundle over $H\unit$ and
$V$ a continuous $H$ action on $Z$.  For $(\gamma,x,\eta)\in G\ltimes
X\rtimes H$ and $(\gamma\inv x\eta, z)\in s_X^*\Zb$ define 
\begin{equation}
(s_X^*V)_{(\gamma,x,\eta)}(\gamma\inv x\eta, z):=(x,V_\eta(z)).
\end{equation}
Then $\{(s_X^*V)_{(\gamma,x,\eta)}\}$ is a continuous $G\ltimes
X\rtimes H$ action on $s_X^*Z$.  In particular, if $(A, \alpha)$ is an
$H$-dynamical system then $(s_X^*A,  s_X^*\alpha)$ is a $G\ltimes
X\rtimes H$-dynamical system. 
\end{prop}

\begin{proof}
This proof is relatively straightforward and the details have been
omitted for brevity.  Algebraic computations show that
$s^*_XV_{(\gamma,x\eta)}$ is an isomorphism and that $s^*_X V$
respects the groupoid operations.  The continuity of $s^*_X V$ follows
from the continuity of $V$ and an application of 
Proposition~\ref{Proposition c.20}.  
\end{proof}

Hence $\nu^{X,H}$ defines a map from $H$-dynamical systems to $G\ltimes
X\rtimes H$-dynamical systems.  We show in the next proposition that
$\nu^{X,H}$ descends to a map on equivariant Morita equivalence
  classes of $H$-dynamical systems. 

\begin{prop}
\label{prop: sX well defined}
Let $X$ be a $(G, H)$-equivalence and $(\Zb, V)$ an equivariant
imprimitivity bimodule bundle between the $H$-dynamical systems
$(A,\alpha)$ and $(B ,\beta)$.  Then $(s_X^*\Zb,
s_X^*V)$ is an equivariant imprimitivity bimodule bundle between the
$G\ltimes X\rtimes H$-dynamical systems  $(s_X^*A, s_X^*\alpha)$ and
$(s_X^*B,  s_X^*\beta)$ where the inner products and
actions are defined as follows 
\begin{align*}\linner{s_X^*A(x)}{(x,z)}{(x,w)}&:=(x,\linner{A(s(x))}{z}{w})&
  \rinner{s_X^*B(x)}{(x,z)}{(x,w)}&:=(x,\rinner{B(s(x))}{z}{w})\\ 
(x,a)\cdot (x,z) &:=(x,a\cdot z) &
(x,z)\cdot(x,b) &:=(x,z\cdot b)
\end{align*}  
for $x\in X$, $z,w\in Z(s_X(x))$, $a\in A(s_X(x))$ and $b\in B(s_X(s))$.
\end{prop}

\begin{proof}			
By the definition of $s_X^*\mathscr{Z}$, each fibre of $s_X^*\Zb$ is
isomorphic as a Hilbert bimodule to a fibre of $\Zb$ and therefore is
an imprimitivity bimodule.  To show that $s_X^*\mathscr{Z}$ is an
imprimitivity bimodule bundle it remains to show that the actions and
inner products are continuous.  However, this follows quickly using
the continuity of the actions on $\Zb$.  Finally,  straightforward
computations show both 
\begin{align*}
(s_X^*\alpha)_{(\gamma,x,\eta)}&(\linner{s_X^*A(\gamma\inv x\eta)}{(\gamma\inv x\eta,z)}{(\gamma\inv x\eta,z')})\\
&=\linner{s_X^*A(x)}{(s_X^*V)_{(\gamma, x,\eta)}(\gamma\inv
  x\eta,z)}{(s_X^*V)_{(\gamma, x,\eta)}(\gamma\inv x\eta,z')}\ \text{ and}\\ (s_X^*\beta)_{(\gamma,x,\eta)}&(\rinner{s_X^*B(\gamma\inv x\eta)}{(\gamma\inv x\eta,z)}{(\gamma\inv x\eta,z')})\\
&=\rinner{s_X^*B(x)}{(s_X^*V)_{(\gamma,
    x,\eta)}(\gamma\inv x\eta,z)}{(s_X^*V)_{(\gamma,
    x,\eta)}(\gamma\inv x\eta,z')}. \qedhere
\end{align*}
\end{proof}	

Proposition~\ref{prop: sX well defined} shows that $\nu^{X,H}$ descends to
a well-defined set map $S(H)\to S(G\rtimes X\ltimes H)$.  Next we show
that $\nu^{X,H}$ is a semigroup homomorphism.  

\begin{prop}
\label{prop sx homo} Let $(A, \alpha)$ and $(B,\beta)$ be
$H$-dynamical systems.  \begin{enumerate}
\item\label{it: iso prop sx homo}  The map $\Phi:s_X^*\A\otimes_X s_X^*\B\to s_X^*(\A\otimes _{H\unit}\B)$ characterized by
\[
\Phi: (x,a)\otimes (x,b)\mapsto (x,a\otimes b)
\]
is an isomorphism intertwining the groupoid 
actions $(s_X^*\alpha)\otimes_X (s_X^*\beta)$ and $s_X^*(\alpha\otimes_{H\unit} \beta)$.
\item \label{it: sx homo 2}We have $[s_X^*A, s_X^*\alpha][s_X^*B,
  s_X^*\beta]=[s_X^*(A\otimes_{H\unit} B),
  s_X^*(\alpha\otimes_{H\unit}\beta)]$.  That is, $\nu^{X,H}$ is a
  homomorphism of Brauer semigroups.  
\end{enumerate}
\end{prop}

\begin{proof}
Recall from Section~\ref{sec: brauer} that 
\begin{align*}
(s_X^*A\otimes_X s_X^*B)(x)&= s_X^*A(x)\otimes s_X^*B(x)
\quad\text{and} \\
A(s_X(x))\otimes B(s_X(x))&= (A\otimes_{H\unit} B)(s_X(x)).
\end{align*}
Furthermore $(x,a)\mapsto
a$ is an isomorphism from $s_X^*A(x)$ to $A(s_X(x))$ and similarly
$s_X^*B(x)\cong B(s_X(x))$ and $s_X^*(A\otimes_{H\unit} B)(x)\cong
(A\otimes_{H\unit} B)(s_X(x))$.    Thus for a fixed $x$ the map
$(x,a)\otimes (x,b)\mapsto (x,a\otimes b)$  is the composition of
isomorphisms  
\begin{align*}
(s_X^*A\otimes_X s_X^*B)(x)&\to s_X^*A(x)\otimes s_X^*B(x)\to A(s_X(x))\otimes B(s_X(x))\\
&\to (A\otimes_{H\unit} B)(s_X(x))\to s_X^*(A\otimes_{H\unit} B)(x).
\end{align*}
Therefore $\Phi$ is an isomorphism on the fibres and hence bijective. 
Thus to show $\Phi$ is an
isomorphism we need to show that $\Phi$ and $\Phi\inv$ are
continuous.  The continuity of $\Phi$ follows from an application of
Proposition~\ref{Proposition c.20}.  The argument is similar to the one
given below and will not be reproduced
here.

To see $\Phi\inv$ is continuous, suppose $(x_i, z_i)\to (x,z)\in
s_X^*(\A\otimes_{H\unit} \B)$.    Let $\epsilon>0$ be given and pick a
finite subset  $I\subset A(s_X(x))\times B(s_X(x))$ so that we have
$\|\sum_{(a,b)\in I} a\otimes b-z\|<\epsilon$. Let $\pi_i$ be the
projection onto the $i$-th factor. For $a\in \pi_1(I)$ and
$b\in \pi_2(I)$ pick functions $f_a\in A $ and $g_b\in B$ such that
$f_a(s_X(x))=a$ and $g_b(s_X(x))=b$.  Now choose a compact neighborhood
$K$ of $x\in X$ and a function $\phi\in C_c(X)$ such that
$\phi|_K\equiv 1$.  The maps $F_{a,b}:y\mapsto (y, \phi(y)
f_a(s_X(y)))\otimes (y,\phi(y) g_b(s_X(y)))$ are continuous and
compactly supported and thus are in $s_X^*A\otimes_X s_X^*B$.  Hence
\[
\sum_{(a,b)\in I} F_{a,b}(x_i)\to \sum_{(a,b)\in I
}F_{a,b}(x)=\sum_{(a,b)\in I} (x,a)\otimes (x, b).
\]

To show $\Phi\inv(x_i,z_i)\to \Phi\inv(x,z)$   it suffices to show  $\Phi\inv(x_i,z_i)$ is
eventually close to $\sum_{(a,b)\in I} F_{a,b}(x_i)$ by Proposition~\ref{Proposition c.20}. Note that since $K$
is a compact neighborhood of $x$  and $\phi|_K\equiv 1$ we eventually have
\[
F_{a,b}(x_i)=(x_i, f_a(s_X(x_i)))\otimes (x_i,
g_b(s_X(x_i)))=\Phi\inv(x_i, f_a(s_X(x_i))\otimes  g_b(s_X(x_i))).
\]
Since $f_a$ and $g_b$ are continuous, for large enough $i$
\[
\Bigg\|\sum_{(a,b)\in
  I}(x_i, f_a(s_X(x_i))\otimes  g_b(s_X(x_i)))-z_i\Bigg\|<\epsilon.
\]
Again since $\Phi$ is an isomorphism of the fibres, we eventually have  
\begin{align*}
\Bigg\|\sum_{(a,b)\in I} F_{a,b}(x_i)&-\Phi\inv(x_i,
z_i)\Bigg\|\\
&=\Bigg\| \sum_{(a,b)\in I} \Phi\inv(x_i,f_a(s_X(x_i))\otimes
g_b(s_X(x_i)))-\Phi\inv(x_i, z_i)\Bigg\|\\ 
&=\Bigg\|\sum_{(a,b)\in I}(x_i,f_a(s_X(x_i))\otimes g_b(s_X(x_i))) -z_i\Bigg\|<\epsilon.
\end{align*}
So Proposition~\ref{Proposition c.20} shows that $\Phi\inv( x_i,
z_i)\to\Phi\inv (x,z)$ and $\Phi\inv$ is
continuous.

It remains to show that $\Phi$ intertwines the actions.  This follows
from a computation on elementary tensors.  
Part~\eqref{it: sx homo 2} of the proposition follows from
part~\eqref{it: iso prop sx homo} since 
\begin{align*}
[s_X^*A, s_X^*\alpha][s_X^*B, s_X^*\beta]&=[s_X^*A\otimes_X
s_X^*B,s_X^*\alpha\otimes_X s_X^*\beta]\\ &=[s_X^*(A\otimes_{H\unit} B),
s_X^*(\alpha\otimes_{H\unit}\beta)]. \qedhere
\end{align*}
\end{proof}

%%%%%%%%%%%%%%%%%%%%%%%%%%%%%%%%%%%%%%%%%%%%%%%%%%%%%%%%%%%%%%%%%%%%%%%%%%%%%%%%%%%%%%%%%%%%%%%%%%%%%%%%%%%%%%%%%%%%%%%%%%%%%%%%%%%%%%%%%%%%%%%%%%%%%%%

\subsection{The Generalized Fixed Point Algebra}
\label{Fix}

The inverse of $\upsilon^{X,H}$ will be constructed using the
generalized fixed point algebra. Let $(A, \omega)$ be a $G\ltimes
X\rtimes H$-dynamical system. Since $G$ acts freely and properly on
$X$, $G\ltimes X$ is a principal and proper groupoid.  Thus we may 
construct the generalized fixed point algebra, $\Fix(A,\omega^G)$
\cite[Proposition~4.4, Remark~3.10]{mep09}.  By
Proposition~\ref{prop: gen fix ind},
$\Fix(A,\omega^G)$ is equal to $\Ind{}{}(A,\omega^G)$.  
We denote both by $\Fix_G(A)$.  
Since $s_X:X\to H\unit$ descends to a homeomorphism of $(G\ltimes
X)\backslash X$ with $H\unit$, by Proposition~\ref{lem:  ind fix}
$\Fix_G(A)$ is a $C_0(H\unit)$-algebra with fibres
$\Fix_G(A)(u)=\Ind{}{}(\Gamma_0(s_X\inv(u), \A),\alpha)$. 

More generally, let $\Zb$ be an upper semicontinuous Banach bundle
over $X$ endowed with a continuous $G\ltimes X\rtimes H$ action
$V=\{V_{(\gamma,x,\eta)}\}$.  
We can define actions $V^G$ and $V^H$ of $G\ltimes X$ and
$X\rtimes H$ on $\Zb$ by restriction.   Let $Z=\Gamma_0(X,\Zb)$.  Define $\Fix_G(Z):=\Ind{}{}(Z,V^G)$.  Consider the sets 
\[
\Fix_G(Z)(u):=\{f\in \Gamma^b(s_X\inv(u), \Zb): V^G_{(\gamma,x)}(f(\gamma\inv x))=f(x)\}.
\]
For $x\in s_X\inv(u)$ the evaluation map $\varepsilon_x:\Fix_G(Z)(u)\to Z(x)$
is isometric since $s_X\inv(u)=(G\ltimes X)\cdot x$
for some $x\in X$ and
$\|f(\gamma x)\|=\|V^G_{(\gamma,x)}(f(x))\|=\|f(x)\|$.  Consequently,
$\varepsilon_x$ has a closed range and it then follows from Lemma~\ref{claim: fibre
  dense} that $\varepsilon_x$ is surjective.  In other words,
$\varepsilon_x:\Fix_G(Z)(u)\to Z(x)$ is a norm preserving isomorphism.
We can then put a topology on $\bigsqcup_{u\in H\unit} \Fix_G(Z)(u)$
using Proposition~\ref{prop: defn a bundle} and the sections $u\mapsto
F|_{s_X\inv(u)}$ for $F\in \Fix_G(Z)$.  Denote  $\bigsqcup_{u\in
  H\unit} \Fix_G(Z)(u)$ equipped with this topology by $\Fix_G(\Zb)$.  Using Proposition~\ref{prop: when dense} and Lemma~\ref{claim: fibre dense}, $\Fix_G(Z)=\Gamma_0(H\unit, \Fix_G(\Zb))$.

\begin{prop}
\label{prop Fix act}
Let $\Zb$ be an upper semicontinuous Banach bundle over $X$ and $V$ a
continuous action of $G\ltimes X\rtimes H$ on $Z$.  For $\eta\in H$
and  $f\in \Fix_G(Z)(s(\eta))$ define 
\begin{equation}
\label{eq Fix act}
\Fix_G(V)_\eta(f)(x):=V^H_{(x,\eta)}(f(x\eta))=V_{(r_X(x),x,\eta)}(f(x\eta)).
\end{equation}
Then $\Fix_G(V)$ is a well-defined continuous action of $H$ on
$\Fix_G(Z)$.  In particular, if $(A,\omega)$ is a $G\ltimes X\rtimes
H$-dynamical system then $(\Fix_G(A), \Fix_G(\omega))$ is an
$H$-dynamical system. 
\end{prop}

\begin{proof}
Suppose $\eta\in H$ and  $f\in \Fix_G(Z)(s(\eta))$.  We first show
$\Fix_G(V)_\eta(f)\in \Fix_G(Z)(r_H(\eta))$. Since $\supp(f)\subset
s_X\inv(s_H(\eta))$ we have $\supp(\Fix_G(V)_\eta(f))\subset
s_X\inv(r_H(\eta))$. We know $\Fix_G(V)_\eta(f)$ is continuous since $V$ and
$f$ are continuous.  It is bounded since $f$ is bounded and
$V_{(r_X(x),x,\eta)}$ is a norm preserving isomorphism for all $x\in
s_X\inv (r_H(\eta))$. Lastly    
\begin{align*}
V^G_{(\gamma,x)}(\Fix_G(V)_\eta(f)(\gamma\inv x))&=V^G_{(\gamma,x)}V^H_{(\gamma\inv x,\eta)}(f(\gamma\inv x\eta))=V^H_{(x,\eta)}V^G_{(\gamma,x\eta)}(f(\gamma\inv x\eta))\\
&=V^H_{(x,\eta)}(f(x\eta))=\Fix_G(V)_\eta(f)(x).
\end{align*}
Thus $\Fix_G(V)_\eta(f)\in \Fix_G(Z)(r_H(\eta))$.  It follows from
routine computations that each $\Fix_G(V)_\eta$ is an isometric
isomorphism and that $\Fix_G(V)$ preserves the groupoid operations.

It remains to show that $\Fix_G(V)$ is continuous.  Suppose $\eta_i\to
\eta_0$ and $f_i\in \Fix(Z)(s(\eta_i))$ such that $f_i\to f_0$.  We
need to show $\Fix_G(V)_{\eta_i}(f_i)\to \Fix_G(V)_{\eta_0}(f_0)$.  It
suffices to show that every subnet has a subnet converging to
$\Fix_G(V)_{\eta_0}(f_0)$.  Pass to a subnet, relabel, and 
pick $x_0\in s_X\inv(r_H(\eta_0))$ and
$F\in \Fix_G(Z)$ such that
$F|_{s_X\inv(r(\eta_0))}=\Fix_G(V)_{\eta_0}(f_0)$.  Since $s_X$ is
open we can choose $x_i\in s_X\inv(r_H(\eta_i))$ such that $x_i\to
x_0$.  
It follows from an application of Proposition~\ref{Proposition c.20} that $f_i(x_i\eta_i)\to f_0(x_0\eta_0)$
in $\Zb$.  Thus by the continuity of $V$,
\begin{align*}
\Fix_G(V)_{\eta_i}(f_i)(x_i)&=V_{(r_X(x_i),x_i,\eta_i)}(f(x_i\eta_i))\to\\
&V_{(r_X(x_0),x_0,\eta_0)}(f(x_0\eta_0))=\Fix_G(V)_{\eta_0}(f_0)(x_0).
\end{align*}
Let $\epsilon > 0$.  
It follows from the definition of the topology on $\Zb$ and the continuity of
$F$ that eventually $\|F(x_i) - \Fix_G(V)_{\eta_i}(f_i)(x_i)\|<\epsilon$.
Since $V^G_{(\gamma,x_i)}$ is a norm preserving isomorphism  we have, for large $i$ and for all $\gamma\in r_G\inv(r_X(x_i))$ , that  
\[
\|V^G_{(\gamma,x_i)}(F(x_i)) -
V^G_{(\gamma,x_i)}(\Fix_G(V)_{\eta_i}(f_i)(x_i))\| < \epsilon.
\] 
  Because $F\in \Fix_G(Z)$ and
$\Fix_G(V)_{\eta_i}(f_i)$ are in $\Fix_G(Z)(r(\eta_i))$, this implies
that $\|F|_{s_X\inv(r_H(\eta_i))}-\Fix_G(V)_{\eta_i}(f_i)\|_\infty <
\epsilon $ eventually.  Another application of Proposition~\ref{Proposition c.20} now shows  $\Fix_G(V)_{\eta_i}(f_i)\to
\Fix_G(V)_{\eta_0}(f_0)$ as desired. 
\end{proof}

We need to show that $\Fix_G$ induces a well-defined map on equivariant Morita equivalence classes.  For this we use the next proposition. 

\begin{prop}
\label{lem: Fix well defined}
Let $X$ be a $(G,H)$-equivalence and $(\Zb, V)$ an equivariant
imprimitivity bimodule bundle for $G\ltimes X\rtimes H$-dynamical
systems  $(A,\alpha)$ and $(B,\beta)$.  Then $(\Fix_G(\Zb), \Fix_G(V))$ is
an equivariant imprimitivity bimodule bundle for $H$-dynamical
systems $(\Fix_G(A), \Fix_G(\alpha))$ and $(\Fix_G(B), \Fix(\beta))$
where the left and right actions and inner products are
defined by
\begin{align*}\linner{\Fix_G(A)(u)}{z}{w}&:=x\mapsto\linner{A(x)}{z(x)}{w(x)}&
\rinner{\Fix_G(B)(u)}{z}{w}&:=x\mapsto \rinner{B(x)}{z(x)}{w(x)}\\
a\cdot z &:= x\mapsto a(x)\cdot z(x) &
z\cdot b &:= x\mapsto z(x)\cdot b(x)
\end{align*}
for $u\in H\unit$, $x\in s_X\inv (u)$, $a\in A$, $b\in B$ and $z,w\in Z$. 
\end{prop}

\begin{proof}
First note that $\Fix_G(V)$ is a continuous action on $\Fix_G(Z)$ by Proposition~\ref{prop Fix act}.
Also $\Fix_G(Z)(u)\cong Z(x)$ for any $x\in s_X\inv (u)$ and the
Hilbert bimodule structure on $\Fix_G(Z)(u)$ is the one pulled back
from $Z(x)$ under this isomorphism. Thus each  fibre  of $\Fix_G(\Zb)$
is an imprimitivity bimodule.  To show that $\Fix_G{\Zb}$ is an
imprimitivity bimodule bundle it suffices to show that the operations
are continuous.  By symmetry it suffices to show that the $\Fix_G(\A)$
action and $\Fix_G(\A)$ inner product are continuous. We show only the
continuity of the $\Fix_G(\A)$ action.  The proof of continuity for
the inner product is similar.

Suppose $a_i \to a_0$ and $z_i \to
z_0$ in $\Fix_G(\A)$ and $\Fix_G(\Zb)$, respectively.  Let $\epsilon >
0$ and define $u_i = p(a_i) = p(z_i)$ for all $i$.  Pick $a \in \Fix_G(A)$ such that
$a|_{s_X\inv(u_0)} = a_0$ and $z\in\Fix_G(Z)$ such that $z|_{s_X\inv(u_0)}=z_0$.  Since
$a|_{s_X\inv(u_i)}$ and $a_i$ both converge to $a_0$ in $\Fix_G(\A)$ we must eventually have
$\|a_i - a|_{s_X\inv(u_i)}\|< \epsilon$.  Similarly, we eventually have
$\|z_i - z|_{s_X\inv(u_i)}\| < \epsilon$.  Notice this also implies that for
large $i$
\[
\|a_i\| \leq \|a_i-a|_{s_X\inv(u_i)}\|+\|a|_{s_X\inv(u_i)}\| \leq \epsilon + \|a\|
\]
so that $\{\|a_i\|\}$ must be bounded by some $M$.  Finally, observe
that 
\[
a|_{s_X\inv(u_i)} \cdot z|_{s_X\inv(u_i)} = (a\cdot z)|_{s_X\inv(u_i)}
\to (a\cdot z)|_{s_X\inv(u_0)} = a_0\cdot z_0.
\]
We may now compute for $x\in s_X\inv(u_i)$,
\begin{align*}
\|(a\cdot z)|_{s_X\inv(u_i)}(x) &- a_i \cdot z_i(x)\|\\
&\leq \|a(x)\cdot z(x) - a_i(x)\cdot z(x)\| + \|a_i(x)\cdot z(x) -
a_i(x)\cdot z_i(x)\| \\
&\leq \|a|_{s_X\inv(u_i)} - a_i\|\|z|_{s\inv_X}\| + \|a_i\|\|z|_{s_X\inv(u_i)}
-z_i\| \leq \epsilon \|z\| + M \epsilon.
\end{align*}
Hence $\|(a\cdot z)|_{s_X\inv(u_i)} - a_i \cdot z_i\|$ is eventually
small and we can now use Proposition~\ref{Proposition c.20} to conclude that
$a_i\cdot z_i\to a\cdot z$.   Finally, the following identities can be
verified with a brief computation:
\begin{align*}
\Fix_G&(\alpha)_\eta(\linner{\Fix_G(A)(s(\eta))}{z}{w})(x)=\linner{\Fix_G(A)(r(\eta))}{\Fix_G(V)_{\eta}(z)}{\Fix_G(V)_{\eta}(w)}(x)
\\
\Fix_G&(\beta)_\eta(\rinner{\Fix_G(B)(s(\eta))}{z}{w})(x)=\rinner{\Fix_G(B)(r(\eta))}{\Fix_G(V)_{\eta}(z)}{\Fix_G(V)_{\eta}(w)}(x).\qedhere
\end{align*}
\end{proof}

%%%%%%%%%%%%%%%%%%%%%%%%%%%%%%%%%%%%%%%%%%%%%%%%%%%%%%%%%%%%%%%%%%%%%%%%%%%%%%%%%%%%%%%%%%%%%%%%%%%%%%%%%%%%%%%%%%%%%%%%%%%%%%%%%%%%%%%%%%%%%%%%%%%%%%%

\subsection{An isomorphism of Brauer Semigroups}
\label{sec: iso equiv}

In this section we show that $\nu^{X,H}$ and $\Fix_G$ are inverses. 
We begin by showing $\nu^{X,H}\circ \Fix_G=\id$. 

\begin{prop}
\label{prop: sx fix=id}
Let $(A, \omega)$ be a $G\ltimes X\rtimes H$-dynamical system.  Then the map characterized by 
\[
\Upsilon(x,f):=f(x)\quad\text{for $f\in \Fix_G(A)(s_X(x))$}
\]
defines an  isomorphism from $s_X^*\Fix_G(\A)$ to $\A$.  Furthermore,
$\Upsilon$ intertwines the action $s_X^*\Fix_G(\omega)$ with
$\omega$. 
\end{prop}

\begin{proof}
Suppose $(A, \omega)$ is a $G\ltimes X\rtimes H$-dynamical system. By \cite[Proposition~1.3]{RW85} 
\[
s_X^* \Fix_G(A)=\Gamma_0(X,
s_X^*\Fix_G(\A))=C_0(X)\otimes_{G\backslash X} \Fix_G(A).
\] 
For the first statement, we define a $C_0(X)$-linear isomorphism
$\tilde{\Upsilon}:s_X^* \Fix_G(A)\to A$ whose associated isomorphism
of upper semicontinuous $C^*$-bundles is $\Upsilon$. Consider the map
\[
\tilde{\Upsilon}: C_0(X)\otimes_{G\backslash X} \Fix_G(A)\to A\quad\text{characterized by}\quad
\tilde{\Upsilon}(\phi\otimes f)(x)=\phi(x)f(x).
\]
Then $\tilde{\Upsilon}$ defines a $C_0(X)$-linear $*$-homomorphism. By
comparing on elementary tensors we see that
$\tilde{\Upsilon}(F)(x)=F(x)(x)=\Upsilon(x,F(x))$ for
$F\in\Gamma_0(X,s_X^*\Fix_G(\A))$.   Therefore the map of
$s_X^*\Fix_G(A)$ induced by $\Upsilon$ is $\tilde{\Upsilon}$. 

Let $B$ be the image of $\tilde{\Upsilon}$.  By definition
$C_0(X)\cdot B\subset B\subset A$.  Pick $x\in X$ and $a\in A(x)$.
Given $\epsilon > 0$, Lemma~\ref{claim: fibre dense} implies that
there exists an $F\in \Fix_G(A)$  with $\|F(x)-a\| < \epsilon$.  Pick
$\phi\in C_c(X)$ such that $\phi(x)=1$ then
$\tilde{\Upsilon}(\phi\otimes F)(x)=F(x)$.  Thus
Proposition~\ref{prop: when dense} implies $B$ is dense in $A$ and
therefore $\tilde{\Upsilon}$ is onto. 

To show that $\tilde{\Upsilon}$ is injective we show it preserves
norms.  If $F\in s_X^*\Fix_G(A)$, then 
\begin{align*}
\|\tilde{\Upsilon}(F)\|&=\sup_{\pi\in \widehat{A}}\|\pi(\tilde{\Upsilon}(F))\|=\max_{x\in X}\sup_{\pi\in \widehat{A(x)}}\|\pi\circ q_x(\tilde{\Upsilon}(F))\|\\
&=\max_{x\in X}\sup_{\pi\in \widehat{A(x)}}\|\pi(F(x)(x))\|
=\max_{x\in X}\sup_{\pi\in \widehat{A(x)}}\|\pi\circ \varepsilon_x(F(x))\|\\
&=\max_{x\in X}\sup_{\pi\in \widehat{A(x)}}\|M(\pi\circ q_x)_{G\cdot x}(F(x))\|.
\end{align*}
By Lemma~\ref{lem: rep cor}, $\{M(\pi\circ q_x)_{G\cdot x}:\pi\in
\widehat{A(x)}\}=(\Fix_G(A)(s_X(x)))\sidehat$, and therefore
$\|\tilde{\Upsilon}(F)\| =\|F\|$.
To see $\Upsilon$ intertwines the actions is a computation which we omit. 
\end{proof}

Next we show $\Fix_G\circ \nu^{X,H}=\id$.

\begin{prop}
\label{prop: fix sx =id}
Let $(A, \beta)$ be an $H$-dynamical system. The map 
\[
\Psi: a\mapsto (x\mapsto (x,a(s_X(x))))
\]
defines an isomorphism from $A$ to $\Fix_G(s_X^*A)$ that intertwines $\beta $ and $\Fix_G(s_X^*\beta)$.
\end{prop}

\begin{proof}
For $a\in A$, the map $u\mapsto a(u)$ is continuous and bounded into
$\A$ so the map $x\mapsto (x,a(s_X(x)))$ is continuous and bounded
into $s_X^*\A$. Furthermore,
\begin{align*}
(s_X^*\beta)^G_{(\gamma,x)}(\gamma\inv
x,a(s_X(\gamma\inv x)))&=s_X^*\beta_{(\gamma,x,s_X(x))}(\gamma\inv
x,a(s_X(\gamma\inv x))) \\
~&=(x,a(s_X(\gamma\inv x)))=(x, a(s_X(x))).
\end{align*}
Since $u\mapsto a(u)$ vanishes at infinity and  $\|(x,
a(s_X(x)))\|=\|a(s_X(x))\|$,  the map $G\cdot x\mapsto \|(x,
a(s_X(x)))\|$ vanishes at infinity too.  That is $x\mapsto
(x,a(s_X(x)))\in \Fix_G(s_X^*A)$.  By definition the map $\Psi$ is
$C_0(H^{(0)})$-linear and maps onto the fibres. Thus by
Proposition~\ref{prop: when dense} $\Psi$ is onto.  The map $\Psi$ is
isometric since both norms are supremum norms.   Thus $\Psi$ is an
isomorphism as desired.   

Since $\Psi$ is an isomorphism of the section algebras it induces an
isomorphism of the upper semicontinuous $C^*$-bundles.  From the
definition of $\Psi$, the corresponding bundle isomorphism sends
$a\in A(u)$ to the map $s_X\inv(u)\to \A$ given by $x \mapsto (x, a)$.  
It follows from a brief computation that the isomorphism is equivariant.
\end{proof}

%%%%%%%%%%%%%%%%%%%%%%%%%%%%%%%%%%%%%%%%%%%%%%%%%%%%%%%%%%%%%%%%%%%%%%%%%%%%%%%%%%%%%%%%%%%%%%%%%%%%%%%%%%%%%%%%%%%%%%%%%%%%%%%%%%%%%%%%%%%%%%%%%%%%%%%

\subsection{Morita equivalence}

Let $E$ be a principal and proper groupoid and $(A,\alpha)$ an $E$-dynamical system.  Theorem~5.2 of \cite{mep09} says that $(A,\alpha)$ is saturated with respect to the subalgebra $C_c(E\unit)\cdot A=\Gamma_c(E\unit,\A)$.  By \cite[Definition~5.1]{mep09}  this means that $\Gamma_c(E\unit,\A)$ with actions and pre-inner products given by 
\begin{align*}
\linner{A\rtimes_{r}
  E}{f}{g}(\gamma,x)&:=f(r(\gamma))\alpha_{\gamma}(g(s(\gamma))^*) \\ \rinner{\Fix(A,\alpha)}{f}{g}( u)&:=\int_E
\alpha_{\gamma}(f(s(\gamma))^*g(s(\gamma)))\;d\lambda_E^{u}(\gamma) 
\\
 F\cdot f(u)&:=\int_E
F(\gamma)\alpha_{\gamma}(f(s(\gamma)))\;d\lambda_E^{u}(\gamma)  \\
f\cdot m(u)&:=f(u)m(u)
\end{align*}
for $F\in \Gamma_c(E, r^*\A)$, $f,g\in \Gamma_c(E\unit,\A)$, and $m\in
\Fix(A,\alpha)$ completes to an $A\rtimes_r E - \Fix(A,\alpha)$
imprimitivity bimodule $\Imp(A,E,\alpha)$.  We will denote
$\Imp(A,E,\alpha)$ by $\Imp(A)$ when the action is clear from context.
Note that since $E$ acts properly on $E\unit$, $E$ is topologically
amenable \cite[Corollary~2.1.7]{A-DR00} and thus measurewise amenable
by \cite[Proposition 3.3.5]{A-DR00} so that $A\rtimes_{r,\alpha}
E=A\rtimes_\alpha E$ \cite[Proposition 6.1.8]{A-DR00}.  Thus we only
need to consider the full crossed products. 

If $C$ is an invariant closed subspace of $E\unit$ then $E|_C$ is also
a principal and proper groupoid.  So $(A(C), E|_C, \alpha)$ is a
saturated proper dynamical system and we get that $\Imp(A(C))$ is an
$A(C)\rtimes E|_C-\Fix(A(C),\alpha)$ imprimitivity bimodule as
above. Define  
\[
\Zb:=\bigsqcup_{E\cdot u\in E\backslash E\unit} \Imp(A(E\cdot u))
\]
 and let $p_{\Zb}:\Zb\to E\backslash E\unit$ be the obvious map.
 Since $\Imp(A(E\cdot u))$ is the completion of the section algebra 
$\Gamma_c(E\cdot u,  \A)$, $\Imp(A)$ is the completion of $\Gamma_c(E\unit,\A)$, and the
 map $f\mapsto f|_{E\cdot u}$ from $\Gamma_c(E\unit,\A)\to
 \Gamma_c(E\cdot u, \A)$ is onto, we can consider
 $\Gamma_c(E\unit,\A)$ as a dense subalgebra of sections of $\Zb$ and
 use $\Gamma_c(E\unit,\A)$ as in Proposition~\ref{prop: defn a bundle} to
 define an upper semicontinuous Banach bundle structure on $\Zb$.  In
 the next proposition we reconcile the imprimitivity bimodules
 $\Imp(A(E\cdot u))$ and $\Imp(A)$ by showing $\Gamma_0(E\backslash
 E\unit, \Zb)\cong \Imp(A)$.

\begin{prop}
\label{prop ibmbundle}
Suppose $E$ is a principal and proper groupoid, $(A,\alpha)$ an $E$-dynamical
system, and $\Zb$ as above.  Then $\Gamma_0(E\backslash E\unit , \Zb)$
is an $A\rtimes E-\Fix(A,\alpha)$ imprimitivity bimodule and  the map
$\iota:\Imp(A)\to \Gamma_0(E\backslash E\unit , \Zb)$ characterized
by 
\[
f\mapsto f|_{E\cdot u}\quad \text{for~}f\in
\Gamma_c(E\unit,\A)\quad\text{and}\quad u\in E\unit
\]
defines an isomorphism of $\Imp(A)$ and  $\Gamma_0(E\backslash E\unit , \Zb)$ as imprimitivity bimodules.
\end{prop}

\begin{proof}
By \cite[Proposition~4.2]{goe:mackey1}, $A\rtimes_\alpha E$ is a
$C_0(E\backslash E\unit)$-algebra, the map $F\mapsto F|_{E\cdot u}$
extends to a surjective homomorphism from $A\rtimes E$ to $A(E\cdot
u)\rtimes E|_{E\cdot u}$, and  $A(E\cdot u)\rtimes
E|_{E\cdot u}$ is isomorphic to $(A\rtimes_{\alpha} E)(E\cdot u)$. Furthermore, by
Proposition~\ref{prop: gen fix ind} we know $\Fix(A(E\cdot u))=
\Ind{}{}(A(E\cdot u))$, which by Proposition~\ref{lem:  ind fix} is
isomorphic to $\Fix(A)(E\cdot u)$. By construction $\Zb$ is an
imprimitivity bimodule bundle. Under the above identifications,
Proposition~\ref{prop: bundle corr} implies that $\Gamma_0(E\backslash
E\unit , \Zb)$  is an $A\rtimes E- \Fix(A)$ imprimitivity bimodule
with actions and inner products given by  
\begin{align*}
\linner{A\rtimes_{\alpha} E}{f}{g}(E\cdot u)(\gamma)&:=f(E\cdot r(\gamma))(r(\gamma))\alpha_{\gamma}(g(E\cdot r(\gamma))(s(\gamma))^*)\\
\rinner{\Fix(A)}{f}{g}(E\cdot u)( u)&:=\int_E \alpha_{\gamma}(f(E\cdot u) (s(\gamma))^*g(E\cdot u)(s(\gamma)))\;d\lambda_E^{u}(\gamma)\\
F\cdot f(E\cdot u)(u)&:=\int_E F|_{E\cdot u}(\gamma)\alpha_{\gamma}(f(E\cdot u)(s(\gamma)))\;d\lambda_E^{u}(\gamma)\\
f\cdot m(E\cdot u)(u)&:=f(E\cdot u)(u)m|_{E\cdot u}(u).
\end{align*}

By the definition of the topology on $\mathscr{Z}$, $\iota(f)\in
\Gamma_0(E\backslash E\unit, \mathscr{Z})$ for all $f\in
\Gamma_c(E\unit, \mathscr{A})$.  Using the Tietze extension theorem for
Banach bundles \cite[Proposition~A.5]{MW08Fell}, $\iota$ maps onto
each fibre of $\Gamma_0(E\backslash E\unit, \mathscr{Z})$ and
therefore $\iota$ is onto by Proposition~\ref{prop: when dense}.
Furthermore,
\begin{align*}
\linner{A\rtimes E}{f}{g}(E\cdot u)(\gamma)&:=f(r(\gamma))\alpha_{\gamma}(g(s(\gamma))^*)=\linner{A(E\cdot u)\rtimes E|_{E\cdot u}}{\iota(f)}{\iota(g)}(\gamma).
\end{align*}
Thus $\iota$ preserves left inner products and therefore is norm
preserving.  It follows that $\iota$ is injective and hence
bijective.  Showing $\iota$ preserves the actions and the right inner
product is similar.  Thus $\iota$ is an isomorphism of imprimitivity
bimodules. 
\end{proof}

Let $(A,\omega)$ be a $G\ltimes X\rtimes H$-dynamical system.  Then
$A\rtimes_{\omega^G}(G\ltimes X)$ is a $C_0(H\unit)$-algebra by
\cite[Proposition~3.5]{BGW12} and there exists an action
$\check{\omega}^H$ of $H$ on $A\rtimes_{\omega^G}(G\ltimes X)$
\cite[Proposition~3.7]{BGW12} characterized by 
\[
\check{\omega}^H_\eta(f)(\gamma,x):=\omega_{(r_X(x),x,\eta)}(f(\gamma,x\eta)).
\]

We use Proposition~\ref{prop ibmbundle} in the next lemma to define an action on $\Imp(A)$ that implements an equivariant Morita equivalence between $(A\rtimes_{\omega^G}(G\ltimes X), \check{\omega}^H)$ and $(\Fix_G(A), \Fix_G(\omega))$.
First, recall that $G\backslash X$ is homeomorphic to $H\unit$ so that
for each $u\in H\unit$ there exists $x\in X$ such that $s_X\inv(u) =
G\cdot x$.  

\begin{lem}
\label{lem equivalence}
For each $\eta\in H$, the map $V_\eta: \Gamma_c(G\cdot x\eta,\A)\to
\Gamma_c(G\cdot x,\A)$ given by $V_\eta(f)(y)=\omega_{(r(y), y,
  \eta)}(f(y\eta))$ extends to an isomorphism of $\Imp(A(G\cdot
x\eta))\to \Imp(A(G\cdot x))$.  Furthermore, $\{V_\eta\}$ defines a
continuous action of $H$ on $\Imp(A)$ such that  
\begin{align}
\label{eq:lemequiv}\linner{A(G\cdot x)\rtimes G|_{G\cdot
    x}}{V_\eta(f)}{V_\eta(g)}&=\check{\omega}^H_\eta(\linner{A(G\cdot
  x\eta)\rtimes G|_{G\cdot x\eta}}{f}{g}) \quad\text{and}\\
\nonumber \rinner{\Fix_G(A)(G\cdot
  x)}{V_\eta(f)}{V_\eta(g)}&=\Fix_G(\omega)_\eta(\rinner{\Fix_G(A)(G\cdot
  x\eta)}{f}{g}). 
\end{align}
\end{lem}

\begin{proof}
Since $f$ is continuous and compactly supported and $\omega$ is a
continuous action $V_\eta(f)$ is continuous and compactly supported
and is thus in
$\Gamma_c(G\cdot x,\A)$.  The two algebraic conditions in
\eqref{eq:lemequiv}  follow
from some mostly painless computations which we omit for brevity.  
It follows from \eqref{eq:lemequiv} that 
\begin{align*}
\|V_\eta(f)\|^2&=\|\linner{(A\rtimes (G\rtimes X))(G\cdot
  x)}{V_\eta(f)}{V_\eta(f)}\|=\|\check{\omega}^H_{\eta}(\linner{(A\rtimes (G\rtimes X))(G\cdot x\eta)}{f}{f})\|\\
&=\|\linner{(A\rtimes (G\rtimes X))(G\cdot x\eta)}{f}{f}\|=\|f\|^2
\end{align*}
so that $V_\eta$ preserves the norm on $\Gamma_c(G\cdot x\eta, \A)$ and
therefore extends to a $*$-homomorphism of $\Imp(A(G\cdot
x\eta))$ into $\Imp(A(G\cdot x))$.  Finally, some more algebra shows that
$V_\eta$ is an isomorphism and it preserves the groupoid
operations.  

To show that $V_\eta$ is an action we need to show that it is
continuous. Suppose that $\eta_i\to \eta_0$ and $z_i \to z_i$ in
$\Zb$.  Let $v_i = r(\eta_i)$ and choose $x_i$ so that $G\cdot x_i =
s_X\inv(v_i)$.  To show that $V_{\eta_i}(z_i)\to V_{\eta_0}(z_0)$ it
suffices to show that every subsequence of $V_{\eta_i}(z_i)$ has a
subsequence converging to $V_{\eta_0}(z_0)$.  It follows from (yet another)
application of Proposition~\ref{Proposition c.20} that, after passing to a
subsequence and relabeling, it
suffices to prove $V_{\eta_i}(F|_{G\cdot x_i})\to V_{\eta_0}(F|_{G\cdot x_0})$ for
all $F\in \Gamma_c(X,\A)$.

So let $F\in \Gamma_c(X,\A)$.  We first suppose that $r(\eta_i)=v_0$
eventually. Then 
\begin{align*}
&\|V_{\eta_i}(F|_{G\cdot x_i})- V_{\eta_0}(F|_{G\cdot x_0})\|^2\\
&=\|\rinner{\Fix_G(A)(v_0)}{V_{\eta_i}(F|_{G\cdot x_i})- V_{\eta_0}(F|_{G\cdot x_0})}{V_{\eta_i}(F|_{G\cdot x_i})- V_{\eta_0}(F|_{G\cdot x_0})}(y)\|\\
\intertext{for any $y\in s_X\inv(v_0)$}
&=\Big\|\int_G \omega^G_{(\gamma,y)}((V_{\eta_i}(F|_{G\cdot x_i})(\gamma\inv y)- V_{\eta_0}(F|_{G\cdot x_0})(\gamma\inv y))^*(V_{\eta_i}(F|_{G\cdot x_i})(\gamma\inv y)\\
&\hspace{2 cm}- V_{\eta_0}(F|_{G\cdot x_0})(\gamma\inv y)))\;d\lambda_G^y(\gamma)\Big\|\\
&=\Big\|\int_G \omega_{(\gamma,y, s_X(y))}(\omega_{(s(\gamma), \gamma\inv y, \eta_i)}(F(\gamma\inv y\eta_i)^*F(\gamma\inv y\eta_i))\\
&\hspace{2 cm}-\omega_{(s(\gamma), \gamma\inv y, \eta_0)}(F(\gamma\inv y\eta_0))^*\omega_{(s(\gamma), \gamma\inv y, \eta_i)}(F(\gamma\inv y\eta_i))\\
&\hspace{2 cm}-\omega_{(s(\gamma), \gamma\inv y, \eta_i)}(F(\gamma\inv y\eta_i)^*)\omega_{(s(\gamma), \gamma\inv y, \eta_0)}(F(\gamma\inv y\eta_0))\\
&\hspace{2 cm}+\omega_{(s(\gamma), \gamma\inv y, \eta_0)}(F(\gamma\inv y\eta_0)^*F(\gamma\inv y\eta_0)))\;d\lambda^y_G(\gamma)\Big\|\\
&\leq\int_G \|\omega_{(s(\gamma), \gamma\inv y, \eta_i)}(F(\gamma\inv y\eta_i))-\omega_{(s(\gamma), \gamma\inv y, \eta_0)}(F(\gamma\inv y\eta_0))\|\|F(\gamma\inv y\eta_i)\|\\
&\hspace{2 cm}+\|\omega_{(s(\gamma), \gamma\inv y,
  \eta_i)}(F(\gamma\inv y\eta_i))\\
  &\hspace{2.5 cm}-\omega_{(s(\gamma), \gamma\inv y,
  \eta_0)}(F(\gamma\inv y\eta_0))\|\|F(\gamma\inv
y\eta_0)\|\;d\lambda^y_G(\gamma).
\end{align*}
The integrand is zero unless either $\gamma\inv y\eta_i$ or
$\gamma\inv y\eta_0$ is in $\supp(F)$.  Since $\{y\eta_i\}$ is compact
and the action of $G$ on $X$ is proper,  $K=\{\gamma:
\{\gamma\inv y\eta_i\}\cap \supp(F)\neq \emptyset\}$ is compact; thus 
\begin{align*}
&\|V_{\eta_i}(F|_{G\cdot x_i})- V_{\eta_0}(F|_{G\cdot x_0})\|^2 \\
&\leq\int_G 2\|\omega_{(s(\gamma), \gamma\inv y, \eta_i)}(F(\gamma\inv
y\eta_i))-\omega_{(s(\gamma), \gamma\inv y, \eta_0)}(F(\gamma\inv
y\eta_0))\|\|F\| \chi_{_K}(\gamma)\;d\lambda^y_G(\gamma). 
\end{align*}
The integral goes to zero since the continuity of $\omega$ implies 
\[
\|\omega_{(s(\gamma), \gamma\inv y,
  \eta_i)}(F(\gamma\inv y\eta_i))-\omega_{(s(\gamma), \gamma\inv y,
  \eta_0)}(F(\gamma\inv y\eta_0))\| \to 0
  \]  
and  $\lambda^y(K)<\infty$.  So in this case
$V_{\eta_i}(F|_{G\cdot x_i})\to V_{\eta_0}(F|_{G\cdot x_0})$.  

Next suppose that $r(\eta_i)\neq v_0$ frequently. Since $r(\eta_i)\to
v_0$, we can choose a subsequence and relabel to assume that
$r(\eta_i)\neq r(\eta_j)$ for $\eta_i\neq \eta_j$.  Let
$C=\{r(\eta_i)\}$ and $D=s_X\inv (C)$. Note that both $C$ and $D$ are
closed since $C$ is compact and $s_X$ is continuous. Define a function
$\iota:D\to\N$ by $\iota(x)=i$ if and only if $s_X(x)=r(\eta_i)$.
Standard arguments show that 
\[
F_0(x):=\omega_{(r(x),x, \eta_{\iota(x)})}(F(x\eta_{\iota(x)}))
\]
is in $\Gamma_c(D,\A)$. By \cite[Proposition~A.5]{MW08Fell}  there exists an
$\mathcal{F}\in\Gamma_c(X,\A)$ such that $\mathcal{F}|_D=F_0$.  By the
definition of the topology on $\mathscr{Z}$, $\mathcal{F}$ is a
continuous section.  So 
\[
V_{\eta_i}(F|_{G\cdot x_i})=F_0|_{s_X\inv(r(\eta_i))}=\mathcal{F}|_{s_X\inv(r(\eta_i))}\to
\mathcal{F}|_{s_X\inv(r(\eta_0))}=V_{\eta_0}(F|_{G\cdot x_0})
\]
and thus $V$ is continuous.   
\end{proof}

The payoff of Lemma~\ref{lem equivalence} is the following theorem,
which gives us an ``imprimitivity'' type result for the map $\Fix_G$.  

\begin{thm}
\label{thm: Morita}
Suppose $G$ and $H$ are second countable locally compact Hausdorff
groupoids with Haar systems and $X$ is a $(G,H)$-equivalence.  Suppose
$(A, \omega)$ is a 
$G\rtimes X\ltimes H$-dynamical system.  Then
$A\rtimes_{\omega}(G\rtimes X\ltimes H)$ is Morita equivalent to
$\Fix_G(A)\rtimes_{\Fix_G(\omega)} H$. 
\end{thm}

\begin{proof}
By Lemma~\ref{lem equivalence}, $V$ is an action on $\Imp(A)$
implementing an equivariant Morita equivalence between the $H$-dynamical systems
$(A\rtimes_{\omega^G} G\ltimes X, \check{\omega}^H)$ and $(\Fix_G(A),
\Fix_G(\omega))$.  Now,
\cite[Section~9.1]{MW08} shows that $(A\rtimes_{\omega^G} (G\ltimes
X))\rtimes_{\check{\omega}^H} H$ is Morita equivalent to
$\Fix_G(A)\rtimes_{\Fix_G(\omega)} H$.  However,  \cite[Theorem
4.1]{BGW12} gives that $(A\rtimes_{\omega^G} (G\ltimes
X))\rtimes_{\check{\omega}^H} H\cong A\rtimes_{\omega}(G\rtimes
X\ltimes H)$ so that  $A\rtimes_{\omega}(G\rtimes X\ltimes H)$ is
Morita equivalent to $\Fix_G(A)\rtimes_{\Fix_G(\omega)} H$ as desired.  
\end{proof}

The main result of the paper now follows quickly.  

\begin{proof}[Proof of Theorem~\ref{thm iso semigroup}]
For item \eqref{H part}, by Proposition~\ref{prop sx homo}, $\upsilon^{X,
  H}$ is a semigroup homomorphism.  By Propositions~\ref{prop: sx
  fix=id} and \ref{prop: fix sx =id}, $\upsilon^{X, H}$ is invertible
and hence an isomorphism.  Theorem~\ref{thm: Morita} shows that
$A\rtimes_{\omega}(G\rtimes X\ltimes H)$ is Morita equivalent to
$\Fix_G(A)\rtimes_{\Fix_G(\omega)} H$ and since $\Fix_G$ is the inverse
of $\upsilon^{X,H}$ this gives the result.  More precisely, given an
$H$-dynamical system $(B,\beta)$ let $A = s^*_XB$ and $\omega =
s^*_X\beta$.  Then by Proposition \ref{prop: sx fix=id}, 
$B\rtimes_\beta H$ is isomorphic to $\Fix_G A\rtimes_{\Fix_G\omega}
H$.  However this algebra is Morita equivalent to 
$A\rtimes_\omega(G\ltimes X\rtimes H) = s^*_X
B\rtimes_{s^*_X \beta}(G\ltimes X\rtimes H)$ by Theorem~\ref{thm: Morita}.
Parts \eqref{G part} and \eqref{sym iso} now follow by symmetry.
\end{proof}

%%%%%%%%%%%%%%%%%%%%%%%%%%%%%%%%%%%%%%%%%%%%%%%%%%%%%%%%%%%%%%%%%%%%%%%%%%%%%%%%%%%%%%%%%%%%%%%%%%%%%%%%%%%%%%%%%%%%%%%%%%%%%%%%%%%%%%%%%%%%%%%%%%%%%%%

\section{The construction from \cite{KMRW98}}
\label{sec:constr-from-kmr}

In this section we reconcile our construction with the one used in
\cite{KMRW98}.  In particular we show that the isomorphism
$\upsilon^X$ described in Theorem~\ref{thm iso semigroup} restricts to
the isomorphism $\phi^X:\Br(H)\to \Br(G)$ described by
\cite[Theorem~4.1]{KMRW98}. 

We define the isomorphism $\phi^X$ here for the convenience of the
reader. Suppose $(A, \beta)$ is an $H$-dynamical system with
associated bundle $\A$.  Then $(x\eta,a)\sim(x,\beta_\eta(a))$ characterizes an equivalence relation on $s_X^*\A$.  Let $\A^X:=s_X^*\A/H$ be the
quotient of $s_X^*\A$ by this equivalence relation. Then $\A^X$ is an
upper semicontinuous $C^*$-bundle over $G\unit$. Denote the image of
$(x,a)$ under this equivalence relation by $[x,a]$ and set
$A^X=\Gamma^0(G\unit, \A^X)$.  Now $\beta^X_\gamma([x,a]):=[\gamma
x,a]$ defines an action of $G$ on $A^X$
\cite[Proposition~2.15]{KMRW98}.  They define 
\[
\phi^X([A,\beta]):=[A^X,\beta^X].
\]

\begin{prop}
\label{lem Fix to AX}
Let $(A, \beta)$ be an $H$-dynamical system.  For $F\in \Fix_G(s_X^*A)$ define $\Theta(F)(u)=[x,F(x)]$ where $x\in
r_X\inv(u)$.  Then $\Theta$ is a well-defined isomorphism from
$\Fix_G(s_X^*A)$ to $A^X$. Moreover, $\Theta$ intertwines the actions
$\Fix_G(s_X^*\beta)$ and $\beta^X$.
\end{prop}

\begin{proof}
To see $\Theta$ is well-defined note that if $x,y\in r_X\inv(u)$
then there exists $\eta\in H$ such that $x\eta=y$; therefore
$[y,F(y)]=[x\eta,F(x\eta)]=[x\eta,\beta_{\eta\inv}(F(x))]=[x,F(x)]$.
Furthermore, the image of $\Theta$ is a $C_0(G\unit)$-subalgebra of
$s_X^*\A/H$. For all $u\in H\unit$ and $x\in s_X\inv(u)$ we have $A^X(u)\cong A(x)$ \cite[page~914]{KMRW98} and
$A(x)\cong \Fix_G(s^*_X\A)(u)$ by Proposition~\ref{lem:  ind fix}; thus
Proposition~\ref{prop: when dense} gives that $\Theta$ is onto.  To
see $\Theta$ is injective, note that if $[x, F(x)]=[x, F'(x)]$ then
there exists $\eta\in H$ such that $(x,F(x))=(x\eta,
\beta_\eta\inv(F'(x)))$.  But since the action of $H$ on $X$ is free,
$\eta=s_X(x)$ and so $F(x)=F'(x)$.  Since this must hold for all $x$
we get $F=F'$.    

To show $\Theta$ is an isomorphism it remains to show that $\Theta$ is continuous and open as a map of upper semicontinuous $C^*$-bundles.  An application of Proposition~\ref{Proposition c.20}, which we omit,  shows that $\Theta$ is continuous.  To see that $\Theta$ is open, suppose $[x_i, a_i]\to [x,a]$.  By making
use of the fact that the quotient map $s_X^*\A\to s_X^*\A/H$ associated to
the continuous action of $H$ on $s_X^*\A$ is open, we can pass to a
subnet and find 
$\eta_i$ such that $(x_i\eta_i, \beta_{\eta_i\inv}(a_i))\to
(x,a)$. Let $\epsilon > 0$.  
Now for each $i$ we can pick $f_i\in \Fix_G(s_X^*A)(r_X(x_i))$ such that
\[
\|f_i(x_i\eta_i)-\beta_{\eta_i\inv}(a_i)\|
=\|f_i(x_i)- a_i\|<\epsilon/2
\]
and in a similar fashion we
choose $f\in\Fix_G(s_X^*A)(r_X(x))$ such that $\|f(x)-a\|<\epsilon/2$.  We
want to show that $f_i\to f$.  Pick $F\in \Fix_G(s_X^*A)$ such that
$F|_{r_X\inv(r_X(x))}=f$.  Now since $F$ is a continuous bounded
section of $s_X^*\A$, $F(x_i\eta_i) - \beta_{\eta_i\inv}(a) \to F(x)-a =
f(x)-a$.  Using that the norm is upper
semicontinuous we  eventually have
$\|F(x_i\eta_i)-\beta_{\eta_i\inv}(a_i)\|<\epsilon/2$.  Thus eventually we have
\begin{align*}
\|f_i(x_i\eta_i\eta)-F(x_i\eta_i\eta)\| &=
\|\beta_{\eta\inv}(f_i(x_i\eta_i)-F(x_i\eta_i))\| 
= \|f_i(x_i\eta_i)-F(x_i\eta_i)\| \\
&\leq
\|f_i(x_i\eta_i)-\beta_{\eta_i\inv}(a_i)\|+
\|F(x_i\eta_i) - \beta_{\eta_i\inv}(a_i)\| < \epsilon. 
\end{align*}
Hence $\|f_i-F|_{r_X\inv(r_X(x_i\eta_i))}\|<\epsilon$ for large $i$.  
Using Proposition~\ref{Proposition c.20} one last time, it follows that 
$f_i\to f\in \Fix_G(s_X^*\A)$ as desired and thus $\Theta$ is open.
A straightforward computation shows that $\Theta$ intertwines the
actions. 
\end{proof}

%%%%%%%%%%%%%%%%%%%%%%%%%%%%%%%%%%%%%%%%%%%%%%%%%%%%%%%%%%%%%%%%%%%%%%%%%%%%%%%%%%%%%%%%%%%%%%%%%%%%%%%%%%%%%%%%%%%%%%%%%%%%%%%%%%%%%%%%%%%%%%%%%%%%%%%

\appendix

\section{General proper dynamical systems}

Let $G$ be a second countable locally compact Hausdorff groupoid with
Haar system $\{\lambda^u\}_{u\in G\unit}$.  Let $(A,\alpha)$ be a
$G$-dynamical system.  For a $*$-subalgebra $A_0$ of $A$ let 
\[
M(A_0)^\alpha:=\{d\in M(A):  A_0d\subset A_0,\ 
\bar{\alpha}_\gamma(d(s(\gamma)))=d(r(\gamma)) ~\forall \gamma\in G\}.
\] 
Recall from \cite[Definition~3.1]{mep09} that  $(A,\alpha)$  is \emph{proper} if there is a dense $*$-subalgebra $A_0$ of $A$ such that 

\begin{enumerate}
\item \label{def prop 1}for all $a,b\in A_0$, the function
  $\linner{E}{a}{b}:\gamma\mapsto
  a(r(\gamma))\alpha_{\gamma}\left(b(s(\gamma))^*\right)$ is
  integrable, and 
\item \label{def prop 2}  for all $a,b\in A_0$, there exists a unique
  element $\rinner{D}{a}{b}\in M(A_0)^\alpha$ such that
\[
(c\cdot \rinner{D}{a}{b})(u)=\int_G
c(r(\gamma))\alpha_{\gamma}\left(a^*b(s(\gamma))\right)d\lambda^u(\gamma)
\quad\text{for all $c\in A_0$.}
\]
 \end{enumerate}
 In this case $E=\cspn\{ \linner{E}{a}{b}: a,b\in A_0\}$ is a
 subalgebra of $A\rtimes_{\alpha,r} G$ Morita equivalent to
 $\Fix(A,\alpha):=\cspn\{\rinner{D}{a}{b}: a,b\in A_0\}$
 \cite[Theorem~3.9]{mep09}.  In \cite{mep09} the question was raised
 as to when $E$ is an ideal of $A\rtimes_{\alpha, r} G$.  We provide a
 condition on $A_0$ in the next proposition guaranteeing that $E$ is
 an ideal. 
 
 \begin{prop}
 \label{prop ideal}
 Let $(A,  \alpha)$ be a proper $G$-dynamical system with respect to
 $A_0$ and $C$ an inductive limit dense $*$-subalgebra of $C_c(G)$.
 Suppose  that $C\cdot A_0\subset A_0$, where the action of $C_c(G)$
 on $A$ is given by 
\begin{equation}
\label{eq:4}
 f\cdot a(u):=\int_G f(\gamma)\alpha_{\gamma}(a(s(\gamma)))\;d\lambda^u(\gamma).
 \end{equation}
Then the subalgebra $E\subset A\rtimes_{\alpha, r} G$ guaranteed by
\cite[Theorem~3.9]{mep09} is  an ideal.\end{prop} 

\begin{proof}
Since $E$ is a $*$-subalgebra of $ A\rtimes_{\alpha, r} G$,  
it suffices to show that $B_0*E_0\subset E_0$ for a dense
subalgebra $B_0\subset A\rtimes_{\alpha, r} G$.  Let $\Omega:
C_c(G)\odot A_0\rightarrow \Gamma_c(G, s^*\A)$ be characterized by 
$\Omega(f\otimes a)(\gamma) = f(\gamma)a(s(\gamma))$.  
By \cite[Proof of Proposition 1.3]{RW85}  the image of $\Omega$ is dense
in $\Gamma_c(G, s^*\A)$ in the inductive limit topology.  Since
$C\subset C_c(G)$ is dense in the inductive limit topology so is
$\Omega(C\odot A_0)$.  By \cite[Lemma
4.3]{MW08}  the map $f\mapsto (\gamma\mapsto
\alpha_{\gamma}(f(\gamma)))$ defines an isomorphism
$\alpha_G:\Gamma_0(G,s^*\A)\rightarrow \Gamma_0(G,r^*\A)$. Thus  
\(
B_0:=\alpha_G\circ \Omega(C\odot A_0)
\)
is dense in $\Gamma_c(G, r^*\A)$ in the inductive limit topology.  

It remains to show $B_0*E_0\subset E_0$.  For $F\in \Gamma_c(G,
r^*\A)$ and $a\in A$ define $F\cdot a(u):=\int_G F(\gamma)\alpha_{\gamma}(a(s(\gamma)))\;d\lambda^u(\gamma)$.
If $F\cdot a\in A_0$ then  $F*\linner{E}{a}{b} = \linner{E}{F\cdot a}{b}$.  Thus it suffices to show $F\cdot a\in A_0$ for all $F\in B_0$ and $a\in A_0$.
By the definition of $B_0$ it suffices to show $\alpha_G\circ \Omega(g\otimes b)\cdot a\in A_0$ for all $g\in C$ and $a,b\in A_0$.  But
 \begin{align*}
 \alpha_G\circ \Omega(g\otimes b)\cdot a(u)&=\int_G \alpha_G\circ \Omega(g\otimes b)(\eta)\alpha_{\eta}(a(s(\eta)))\;d\lambda^u(\eta)\\
&=\int_G g(\eta)\alpha_\gamma(ba(s(\eta)))\;d\lambda^u(\eta)
=(g\cdot ba)(u)\in A_0,
\end{align*}
since $ba\in A_0$ and $C\cdot A_0\subset A_0$ by assumption. Hence
$B_0 * E_0\subset E_0$ and thus $E$ is an ideal in $A\rtimes_{\alpha,r}G$ as desired.\end{proof}
 
\begin{rmk}
Let $\mathcal{G}$ be a group.  In \cite{Rie88}, if $(A,\alpha)$ is a
\emph{proper} $\mathcal{G}$-dynamical system with respect to the dense
subalgebra $A_0$, the condition $\alpha_s(A_0)\subset A_0$ for all
$s\in \mathcal{G}$ ensures that $E$ is an ideal in $A\rtimes_{\alpha,
  r} \mathcal{G}$.  As 
observed in \cite{mep09}, $\alpha_s(A_0)\subset A_0$ does not make
sense for groupoids.  It is unclear if the condition of
Proposition~\ref{prop ideal} reduces to the condition that
$\alpha_s(A_0)\subset A_0$ for all $s\in \mathcal{G}$ in the group
case.  However, 
the examples below show that  many proper group dynamical
systems satisfy the condition of Proposition~\ref{prop
  ideal}.
\end{rmk}

\begin{ex}
Suppose $\mathcal{G}$ acts
properly on $X$, $(A, \alpha)$ is a $\mathcal{G}$-dynamical system, and  $\theta: C_0(X)\to M(A)$ is equivariant and nondegenerate. Then
$(A,\alpha)$ is proper with respect to the subalgebra
$A_0=\theta(C_c(X))A\theta(C_c(X))$ \cite[Theorem~5.7]{Rieff04}.  It is easy to see
that $C_c(\mathcal{G})\cdot A_0\subset A_0$. Indeed, suppose $f\in C_c(\mathcal{G})$, $h,k\in
C_c(X)$ and $a\in A$, then we have 
\[
f\cdot (\theta(h)a\theta(k))=\int_\mathcal{G}
f(s)\alpha_s(\theta(h)a\theta(k))\;ds.
\]  
Pick $c\in C_c(X)$ such that
$c\equiv 1$ on the set $\supp(f)\cdot (\supp h\cup \supp k)$.  Then for all
$s\in \supp (f)$, $\theta(c)\alpha_s(\theta(h))=\alpha_s(\theta(h))$ and
$\alpha_s(\theta(k))\theta(c)=\alpha_s(\theta(k))$.  Thus  
\[
f\cdot (\theta(h)a\theta(k))=\theta(c)(f\cdot (\theta(h)a\theta(k)))\theta(c)\in A_0.
\]
\end{ex}

\begin{ex}
Let $\mathcal{G}$ be a compactly generated Abelian Lie group.  Using
\cite[no. 11]{weilunitaire} we know that $\mathcal{G}$ is of the form
$\mathbb{R}^p\times \mathbb{Z}^q \times \mathbb{T}^m \times F$ where
$F$ is a finite Abelian group.  Now let $\mathcal{G}$ act on a $C^*$-algebra $A$
with action $\alpha$.  Let $\widehat{\mathcal{G}}$ be the Pontryagin dual of $\mathcal{G}$
and let $\hat{\alpha}_\omega(f)(s) = \omega(s)f(s)$ be the dual action
in the sense of Takesaki-Takai.  
It follows from the statement and proof of \cite[Theorem
2.2]{Rie88} that the action of $\widehat{\mathcal{G}}$ on
$A\rtimes_\alpha \mathcal{G}$ is proper.  The role of the dense subalgebra $A_0$
is played by the collection $S_\alpha(\mathcal{G},A)$ which is defined as
follows.  Let $\beta$ be the strongly continuous action of $\mathcal{G}\times \mathcal{G}$
on $C_0(\mathcal{G},A)$ by 
\(
\beta_{(s,t)}(f)(r) = \alpha_s(f(r-t)).
\)
Then $S_\alpha(\mathcal{G},A)$ 
is the space of elements of $C_0(\mathcal{G},A)$ which are
infinitely differentiable for the action $\beta$ and which vanish more
rapidly at infinity than any polynomial on $\mathcal{G}$ grows.  Here
derivatives are taken in the $\mathbb{R}$ and $\mathbb{T}$ directions
of $\mathcal{G}$, whereas polynomials are taken with respect to the $\mathbb{R}$
and $\mathbb{Z}$ directions.  

Consider the action of $C_c(\mathcal{G})$ on
$S_\alpha(\mathcal{G},A)$ given by \eqref{eq:4}.  
Observe that there is only one fiber and the Haar system is given by the dual
Haar measure.  After passing evaluation at $s\in \mathcal{G}$ through the
integral we see that for $\phi \in C_c(\mathcal{G})$ and $f\in S_\alpha(\mathcal{G},A)$
\begin{equation}
\label{eq:5}
\phi\cdot f(s) = \int_{\widehat{\mathcal{G}}}
\phi(\omega)\hat{\alpha}_\omega(f)(s)\,d\omega
= \int_{\widehat{\mathcal{G}}} \phi(\omega)\omega(s)\,d\omega f(s)
= \hat{\phi}(s)f(s).
\end{equation}
Here $\hat{\phi}$ denotes the Fourier transform of $\phi$ from an
element of $C_c(\widehat{\mathcal{G}})$ to an element of $C_0(\mathcal{G})$.  

Let $C$ be the set of smooth, compactly supported functions in
$C_c(\widehat{\mathcal{G}})$.  It is not difficult to see that $C$ is dense
with respect to the inductive limit topology.  We wish to show
that $C\cdot S_\alpha(\mathcal{G},A) \subset S_\alpha(\mathcal{G},A)$.  However, in
light of \eqref{eq:5} it suffices to show that
if $\phi \in C$ then $\hat{\phi}$ is
infinitely differentiable in the $\mathbb{R}$ and $\mathbb{T}$
coordinates, and that in the $\mathbb{R}$ and $\mathbb{Z}$
coordinates $\hat{\phi}$ vanishes at infinity faster than
any polynomial grows.  Since the Pontryagin dual of
a product is the product of the Pontryagin duals, and since our
notions of smoothness and growth are all taken coordinatewise, we
need to prove that
\begin{enumerate}
\item if $\phi$ is a compactly supported smooth function on $\mathbb{R}$ then
  $\hat{\phi}$ is a smooth function on $\mathbb{R}$ which vanishes at infinity
  faster than any polynomial on $\mathbb{R}$ grows,
\item if $\phi$ is a smooth function on $\mathbb{T}$ then $\hat{\phi}$ vanishes at
  infinity faster than any polynomial on $\mathbb{Z}$ grows, and 
\item if $\phi$ is finitely supported on $\mathbb{Z}$
  then $\hat{\phi}$ is smooth on $\mathbb{T}$. 
\end{enumerate}  
However, these are all standard facts from Fourier analysis 
\cite[Theorem~2.6, Theorem~7.5]{fourierfolland}.  Thus the conditions
of Proposition \ref{prop ideal} are satisfied in this example.
\end{ex}

%%%%%%%%%%%%%%%%%%%%%%%%%%%%%%%%%%%%%%%%%%%%%%%%%%%%%%%%%%%%%%%%%%%%%%%%%%%%%%%%%%%%%%%%%%%%%%%%%%%%%%%%%%%%%%%%%%%%%%%%%%%%%%%%%%%%%%%%%%%%%%%%%%%%%%%

%\bibliographystyle{amsalpha}
%\bibliography{masterbib}

\begin{thebibliography}{00}
%{KMRW98}

\bibitem
%[ADR00]
{A-DR00}
C.~Anantharaman-Delaroche and J.~N.~Renault, \emph{Amenable groupoids},
  Monographies de L'Enseignement Math\'ematique [Monographs of L'Enseignement
  Math\'ematique], vol.~36, L'Enseignement Math\'ematique, Geneva, 2000, With a
  foreword by Georges Skandalis and Appendix B by E. Germain.

\bibitem
%[Bro12]
{mep09}
J.~H. Brown, \emph{Proper actions of groupoids on ${C}^*$-algebras},
  J. Operator Theory \textbf{67} (2012), no.~2, 437--467.

\bibitem
%[BGW]
{BGW12}
J.~H. Brown, G. Goehle, and D. P. Williams, \emph{Groupoid
  equivalence and the associated iterated crossed product}, to appear, 	arXiv:1206.2066v1 [math.OA].

\bibitem%[Bla96]
{Blan96}
{\'E}. Blanchard, \emph{D\'eformations de {$C\sp *$}-alg\`ebres de
  {H}opf}, Bull. Soc. Math. France \textbf{124} (1996), no.~1, 141--215.
  
\bibitem%[CKRW97]
{CKRW97}
D. Crocker, A. Kumjian, I. Raeburn, and D. P. Williams, \emph{An equivariant
Brauer group and actions of groups on $C^*$-algebras}, J. Funct. Anal. \textbf{146 }(1997), 151--184.

\bibitem%[EW98]
{EW98}
S. Echterhoff and D.~P. Williams, \emph{Crossed products by {$C\sb
  0(X)$}-actions}, J. Funct. Anal. \textbf{158} (1998), no.~1, 113--151.
  

\bibitem%[FD88]
{FD88n1}
J.~M.~G. Fell and R.~S. Doran, \emph{Representations of {$\sp *$}-algebras,
  locally compact groups, and {B}anach {$\sp *$}-algebraic bundles. {V}ol. 1},
  Pure and Applied Mathematics, vol. 125, Academic Press Inc., Boston, MA,
  1988, Basic representation theory of groups and algebras. 

\bibitem%[Fol92]
{fourierfolland}
G.~B. Folland, \emph{Fourier analysis and its applications}, American
  Mathematical Society, 1992.

\bibitem%[Goe10a]
{goe:mackey1}
G. Goehle, \emph{The {M}ackey machine for crossed products by regular
  groupoids. I}, Houston J. of Math. \textbf{36} (2010), no.~2, 567--590.
  
\bibitem%[Goe11]
{goe:unitary}
\bysame, \emph{Locally unitary groupoid crossed products},
Indiana Univ. Math. J. \textbf{60} (2011), no.~2, 411--442.

\bibitem%[Goe10b]
{goe:mackey2}
\bysame, \emph{The {M}ackey machine for crossed products by regular
  groupoids. II}, Rocky Mountain J. of Math., to appear,	arXiv:0908.1434v2 [math.OA].


\bibitem%[HRW00]
{aHRW00}
A.~an Huef, I. Raeburn, and D.~P. Williams, \emph{An equivariant
  {B}rauer semigroup and the symmetric imprimitivity theorem}, Trans. Amer.
  Math. Soc. \textbf{352} (2000), no.~10, 4759--4787. 

%\bibitem[Kas88]{Kas88} Gennadi G. Kasparov, \emph{Equivariant KK-theory and the Novikov conjecture}, Invent. Math. \textbf{91}
%(1988), 147--201.

\bibitem%[KMRW98]
{KMRW98}
A. Kumjian, P.~S. Muhly, J.~N. Renault, and Dana~P. Williams,
  \emph{The {B}rauer group of a locally compact groupoid}, Amer. J. Math.
  \textbf{120} (1998), no.~5, 901--954. 

\bibitem%[KRW96]
{KRW96}
 A. Kumjian, I. Raeburn, and D. P. Williams, \emph{The equivariant Brauer groups of commuting free and
proper actions are isomorphic}, Proc. Amer. Math. Soc. \textbf{124} (1996), 809–-817.

\bibitem%[MRW87]
{MRW87}
P.~S. Muhly, J.~N. Renault, and D.~P. Williams, \emph{Equivalence and
  isomorphism for groupoid {$C\sp \ast$}-algebras}, J. Operator Theory
  \textbf{17} (1987), no.~1, 3--22. 

\bibitem%[MW95]
{MW95}
P.~S. Muhly and D.~P. Williams, \emph{Groupoid cohomology and the Dixmier-Douady class.} 
Proc. London Math. Soc. (3) \textbf{71} (1995), no.~1, 109–-134. 

\bibitem%[MW08a]
{MW08Fell}
\bysame, \emph{Equivalence and disintegration
  theorems for {F}ell bundles and their {$C^*$}-algebras}, Dissertationes Math.
  \textbf{456} (2008), 1--57.

\bibitem%[MW08b]
{MW08}
\bysame, \emph{Renault's equivalence theorem for groupoid crossed products},
  New York Journal of Mathematics Monographs \textbf{3} (2008), 1--87.
  
\bibitem%[Rae88]
{Rae88}
I. Raeburn, \emph{Induced $C^*$-algebras and a symmetric imprimitivity theorem}, Math. Ann. \textbf{280}
(1988), 369--387.

\bibitem%[RW85]
{RW85}
I. Raeburn and D.~P. Williams, \emph{Pull-backs of {$C\sp \ast$}-algebras
  and crossed products by certain diagonal actions}, Trans. Amer. Math. Soc.
  \textbf{287} (1985), no.~2, 755--777. 

\bibitem%[RW98]
{tfb}
\bysame, \emph{Morita equivalence and continuous-trace {$C\sp *$}-algebras},
  Mathematical Surveys and Monorgaphs, vol.~60, American Mathematical Society,
  Providence, RI, 1998. 

\bibitem%[Rie82]
{Rie82}
M. A. Rieffel, \emph{Applications of strong Morita equivalence to transformation group $C^*$-
algebras}, Operator Algebras and Applications (Richard V. Kadison, ed.), Proc. Symp. Pure
Math., vol. 38, Part I, Amer. Math. Soc., Providence, R.I., 1982, pp. 299--310.

\bibitem%[Rie88]
{Rie88}
\bysame, \emph{Proper actions of groups on $C^*$-algebras}, Mappings of operator algebras (H. Araki
and R. V. Kadison, eds.), Progr. Math., vol. 84, Birkhauser, Boston, 1988, Procceedings of
the Japan-U.S. joint seminar, University of Pennsylvania, pp. 141--182.

\bibitem%[Rie04]
{Rieff04}
\bysame, \emph{Integrable and proper actions on {$C\sp *$}-algebras,
  and square-integrable representations of groups}, Expo. Math. \textbf{22}
  (2004), no.~1, 1--53.

\bibitem%[Sed86]
{Sed86}
A.~K. Seda, \emph{On the continuity of {H}aar measure on topological
  groupoids}, Proc. Amer. Math. Soc. \textbf{96} (1986), no.~1, 115--120.
  
\bibitem%[Wei64]
{weilunitaire}
A.~Weil, \emph{Sur certains groupes d'op\`erateurs unitaire}, Acta Mathematica
  \textbf{111} (1964), 143--211.

\bibitem%[Wil07]
{TFB2}
D.~P. Williams, \emph{Crossed products of {$C{\sp \ast}$}-algebras},
  Mathematical Surveys and Monographs, vol. 134, American Mathematical Society,
  Providence, RI, 2007. 

\end{thebibliography}
%\end{document}

%%%%%%%%%%%%%%%%%%%%%%%%%%%%%%%%%%%%%%%%%%%%%%%%%%%%%%%%%%%%%%%%%%%%%%%%%%%%%%%%%%%%%%%%%%%%%%%%%%%%%%%%%%%%%%%%%%%%%%%%%%%%%%%%%%%%%%%%%%%%%%%%%%%%%%%

\providecommand{\bysame}{\leavevmode\hbox to3em{\hrulefill}\thinspace}

\providecommand{\MRhref}[2]{%
  \href{http://www.ams.org/mathscinet-getitem?mr=#1}{#2}
}
\providecommand{\href}[2]{#2}

\end{document}